\documentclass[12pt]{amsart}

\usepackage[utf8]{inputenc}
\usepackage[english]{babel}
\usepackage{amsmath,amssymb,amsfonts,amsthm}
\usepackage{graphicx}
\usepackage{hyperref}
 \usepackage[draft]{fixme}
\newcounter{thm}

\numberwithin{thm}{section}

\newtheorem{theorem}[thm]{Theorem}
\newtheorem*{theorem*}{Theorem}
\newtheorem{lemma}[thm]{Lemma}
\newtheorem{proposition}[thm]{Proposition}
\newtheorem*{conjecture*}{Conjecture}
\newtheorem{corollary}[thm]{Corollary}

\newtheorem{remark}[thm]{Remark}
\theoremstyle{definition}
\newtheorem{definition}[thm]{Definition}

\numberwithin{equation}{section}

\newcommand{\diam}{\operatorname{diam}}

\renewcommand{\Re}{\mathop{\mathrm{Re}}}
\renewcommand{\Im}{\mathop{\mathrm{Im}}}
\newcommand{\eps}{\varepsilon}
\newcommand{\id}{\, id \,}

\newcommand{\bbC}{\mathbb C}
\newcommand{\CC}{\mathbb C}
\newcommand{\bbR}{\mathbb R}
\newcommand{\RR}{\mathbb R}
\newcommand{\bbQ}{\mathbb Q}

\newcommand{\bbN}{\mathbb N}

\newcommand{\bbZ}{\mathbb Z}
\newcommand{\ZZ}{\mathbb Z}

\newcommand{\cD}{\mathcal D}
\newcommand{\cI}{\mathcal I}
\newcommand{\cH}{\mathcal H}
\newcommand{\cR}{\mathcal R}
\newcommand{\cU}{\mathcal U}
\newcommand{\cV}{\mathcal V}

\newcommand{\cP}{\mathcal P}
\newcommand{\rot}{\mathrm{rot}\,}
\newcommand{\dist}{\mathrm{dist}\,}
\newcommand{\dis}{\mathrm{dis}\,}
\renewcommand{\mod}{\operatorname{mod}}
\newcommand{\rdif}{\cR_{\text{dif}}}
\newcommand{\cren}{\cR_{\text{cyl}}}
\newcommand{\crenn}{\cR_{\text{cyl},n}}
\newcommand{\crenN}{\cR_{\text{cyl},N}}
\newcommand{\deps}{\cD_\eps}
\newcommand{\dcreps}{\cD^{\text{cr}}_\eps}
\newcommand{\depsr}{\cD_\eps^\RR}
\newcommand{\dcrepsr}{\cD^{\text{cr},\RR}_\eps}
\newcommand{\tl}{\tilde}
\newcommand{\Ker}{\mathrm{Ker}\,}

\title[Renormalization and Arnold tongues]{Renormalization of circle maps and smoothness of Arnold tongues}
\author{N. Goncharuk and M. Yampolsky}
\begin{document}
\maketitle
\begin{abstract}
 We study the global behavior of the renormalization operator on a specially constructed Banach manifold that has cubic critical circle maps on its boundary and circle diffeomorphisms in its interior. As an application, we prove results on smoothness of irrational Arnold tongues.
\end{abstract}

\section{Introduction}
Arnold tongues are one of the most familiar images in one-dimensional dynamics (Figure~\ref{fig-tongues}).
Consider the two real parameter family of standard (or Arnold) maps
$$f_{a,b}(z)=z+a+\frac{b}{2\pi}\sin 2\pi z \mod\bbZ,\; a,b\in[0,1].$$
\begin{figure}[h!]
  \centerline{\includegraphics[width=0.4\textwidth]{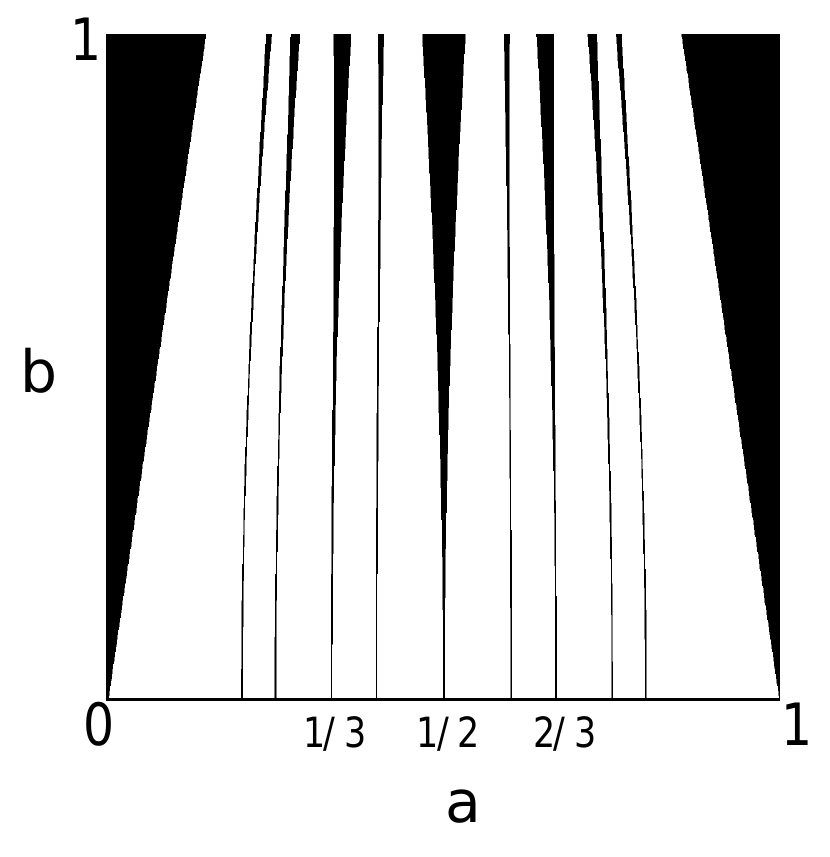}}
  \caption{\label{fig-tongues}Arnold tongues.}
  \end{figure}
They are analytic diffeomorphisms of the circle $\RR/\ZZ$ for $b\in[0,1)$, and analytic homeomorphisms with a single cubic critical point (critical circle maps) for $b=1$. The rotation number $\rot (f_{a,b})$ is a non-decreassing function of $a$. For each fixed $b>0$, the graph of the  function
  $$a\mapsto\rot(f_{a,b})$$
  is a ``devil's staircase'': a continuous non-decreasing curve with flat ``steps'' at rational heights. By definition, Arnold $\alpha$-tongue is the set of parameters $(a,b)$ which corresponds to a rotation number $\alpha$. For a rational $\alpha=p/q$, this set is bounded by two curves (graphs over the $b$-axis) which correspond to the algebraic condition of having a periodic orbit with combinatorial rotation number $p/q$ and the unit multiplier. The boundaries of a rational tongue are thus algebraic. They meet at the point $(p/q,0)$ corresponding to the rigid rotation by $p/q$, cutting out a sharp tongue-looking shape. For an irrational $\alpha$, the tongue is a curve -- a continuous graph over the second coordinate. The  question of smoothness of irrational tongues is deep and fundamental; it is the central theme of this paper.

In 2001, L.~Slammert \cite{Slammert} proved $C^1$-smoothness of \emph{open} irrational Arnold tongues (that is, without the endpoints $b=1$ corresponding to critical circle maps). His argument rests on the theorem of Douady and Yoccoz  that for a circle diffeomorphism $f$ there exists a unique  {\it (-1)-measure} $\lambda$ which satisfies the invariance law
$$\int\xi(f(x))\frac{1}{f'(x)}d\lambda=\int\xi d\lambda\text{ for a test function }\xi.$$
Slammert used these measures to describe the tangent bundle of an irrational tongue.

 Much stronger results can be obtained by imposing arithmetic conditions on $\alpha$. The  strongest of them is due to E.~Risler \cite{Risler} who showed that open Arnold tongues corresponding to Herman numbers $\alpha$ are analytic curves. The Herman class (defined by Yoccoz \cite{Yoccoz2002})  consists of rotation numbers $\alpha$ such that if $\rot f = \alpha$ for an analytic  circle diffeomorphism $f$, then it is analytically conjugate to the rotation by $\rot f$.
Risler also showed that in a larger Brjuno class of rotation numbers, the curves are locally analytic for sufficiently small values of $b$, and, furthermore, form a foliation by analytic curves over $0\leq b<b_0$ given a Brjuno condition with a uniform rate of convergence (see \cite{Risler} for the details). Using quasiconformal surgery, similar results were later obtained in the standard family in \cite{FaGe}.


%

  Note that the above results do not address the question of the degree of  smoothness of Arnold tongues at the ends $b=1$, which is quite subtle. Indeed, the above proofs break down when circle maps develop critical points.
  In \cite{LlaveLuque}, De la Llave and Luque
   performed numerical experiments on the smoothness of Arnold tongues at  $b=1$.
  Based on the results of these experiments, they conjectured that Arnold tongues that correspond to Diophantine rotation numbers are finitely smooth at the ends. In particular, they conjectured that the tongue that corresponds to the golden ratio is $C^2$-smooth at the endpoint. De la Llave and Luque suggested an explanation of this fact that involves the behaviour of a conjectural renormalization operator on a neighborhood of critical circle maps. 

In the same way as for the standard family, given any parametric family $f_\mu$, $\mu\in\bbR^n$ of circle homeomorphisms, the Arnold $\alpha$-tongue can be defined as the set of parameters $\mu $ for which $\rot(f_\mu)=\alpha$. The results of Slammert and Risler, as well as the conjectures of De la Llave and Luque suitably translate into this general setting.

Recently, in \cite{GY}, we developed a new approach to results of \cite{Risler} by constructing an analytic renormalization operator for which Brjuno rotations form a hyperbolic invariant set with a codimension-one stable foliation. Furthermore, the analytic stable submanifolds coincide with analytic conjugacy classes of rotations. For an analytic family of diffeomorphisms which crosses an $\alpha$-leaf of this foliation transversally, the Arnold $\alpha$-tongue will be an analytic codimension-one surface, implying the above quoted results of Risler. In the present paper, we define a novel renormalization framework, which combines the renormalization of diffeomorphisms and renormalization of critical circle maps developed by the second author (see \cite{Ya3,Ya4} and references therein).

The new renormalization operator has two hyperbolic horseshoes: the one consisting of rigid rotations with a single unstable direction as in \cite{GY} and another ``nontrivial'' one consisting of critical circle maps with two unstable directions. Arnold tongues in this picture lie in the stable-unstable manifolds, whose smoothness at critical circle maps is quantified in terms of the Lyapunov exponents of the second horseshoe. This is, roughly, what was pictured in \cite{LlaveLuque}. We give  estimates of the expansion factors of renormalization of critical circle maps, and use this to give a lower bound on the smoothness of closed Arnold tongues.

Of course, renormalization of diffeomorphisms we defined in \cite{GY} cannot be directly extended to maps with critical points. The method we use to put the two horseshoes under one roof is new and will likely be useful in other contexts. Another notable step in our construction is extension of the results on existence, uniqueness, and general properties of (-1)-measures to maps with critical points. This is a necessary part of our proofs, and also allows us to extend Slammert's result to closed, rather than open, Arnold $\alpha$-tongues for {\it all} irrational $\alpha$. But it is also of an independent interest: (-1)-measures are a useful technical tool in the study of dynamics, but also provide a new description of the stable tangent bundle of renormalization. In \S~\ref{sec-unifhyp} we use (-1)-measures to give a new proof of renormalization expansion, with explicit bounds.
Another useful tool developed in this paper is a theorem on the smoothness of stable-unstable manifolds in a general Banach space setting, which has been previously missing in the literature.

  

Let us proceed with formulating our main results. 
Let $\deps$ be the set of analytic maps $f\colon \bbC/\bbZ \to \bbC/\bbZ$ that have bounded analytic continuations to the strip of width $\eps$ around $\bbR/\bbZ$. Equipped with the sup-norm in this strip, $\deps$ is a complex Banach manifold. Its real slice $\depsr$ consists of circle-preserving maps. Let $\dcreps \subset \deps$ be the Banach submanifold of analytic maps  $f\colon \bbC/\bbZ \to \bbC/\bbZ$  with $f'(0)=0$, $f''(0)=0$, $f'''(0)> 0$, that extend to the strip of width $\eps$ around $\bbR/\bbZ$. Its real slice $\dcrepsr=\dcreps\cap \depsr$ consists of {\it critical circle maps} -- analytic circle homeomorphisms with a single cubic critical point at the origin. 

In \cite{Ya3}, the second author constructed the cylinder renormalization operator $\mathcal R_{cyl}$ on an open neighborhood of $\dcrepsr$ in $\dcreps$ for sufficiently large $\eps$, and showed that it is hyperbolic, with one unstable direction, on an invariant horseshoe-like set $\mathcal I$. This result is formulated below, see Theorem \ref{th-horseshoe}.

In \cite{GY}, we defined a renormalization operator $\rdif$ on a neighborhood of rotations in  the space $\deps$ for sufficiently large $\eps$ (the space of analytic diffeomorphisms of an annulus), and proved its hyperbolicity at Brjuno rotations, with a single unstable direction corresponding to the rotation angle. The construction was motivated by Risler's result \cite{Risler}, and by  Yoccoz's result \cite{Yoc} on linearizations of circle diffeomorphisms.

In both cases, the renormalization of a circle map was defined to be the first return map to  a fundamental domain $[0, f^n(0)]$ for a suitably chosen $n$, in a certain analytic chart on $[0, f^n(0)] / f^n \sim  \bbR/\bbZ$. The key to either construction lied in the choice of a specific analytic chart. 
Let us generally say that a smooth mapping $\cR$ of $\depsr$ to itself is a renormalization operator if for any circle map $f$ in its domain there exists $N$ such that the circle map $\cR f$ is conjugate, via an analytic circle map, to a first-return map under $f$ to a quotient of the fundamental domain $[0, f^{N}(0)] / f^N$.

\begin{theorem}
\label{th-operator}
For a sufficiently large $\eps$, there exists a Banach manifold  $D$, a codimension-1 submanifold $D_0\subset D$, a projection $p \colon D\to \deps$ (defined on a certain subset of $D$) such that $p(D_0)=\dcreps$, and an operator $\mathcal R \colon D\to D$ such that the following holds:
 \begin{itemize}
  \item $\mathcal R$ is an analytic operator with compact derivative on its domain, and preserves $D_0$;
  \item $\mathcal R$ commutes with the projection: $p \mathcal R = \mathcal R p$ on the respective domains of definition;
  \item $\mathcal R$ contracts on fibers $p^{-1} f$;
  \item The restriction of $p \mathcal R p^{-1}$ to circle homeomorphisms  is a renormalization operator, in the sense described above;
  \item There exists a set $\Lambda \subset D_0$ invariant under $\mathcal R$ such that $p \Lambda = \mathcal I$, $\mathcal R$ is uniformly hyperbolic on a set $\Lambda\subset D_0$, with two-dimensional unstable subspaces at any trajectory in $\Lambda$. These unstable spaces intersect $D_0$ on one-dimentional subspaces.
  
 \end{itemize}

\end{theorem}

 The Banach manifold $D$ will be the space of triples described in Sec. \ref{sec-triples} below; a similar construction  was used before in \cite{GorYa}.

Using this result, we will prove the following conditional result on the smoothness of Arnold tongues.
\begin{definition}
Let $R$ be a smooth operator in the Banach space. Suppose that for some orbit $g_n = R^n g$, $R$ has an invariant direction field $l_{g_n}\in T_{g_n}D$, that is, $d|_{g_n}R \, l_{g_n} = l_{g_{n+1}}$. We will say that \emph{the maximal (resp. minimal) expansion rate} of $R$ along $l_{g_n}$ is
$$\lambda_{max}=\limsup_{n\to \infty} \sup_{m\ge 0} \sqrt[n]{\|dR^n|_{l_{g_m}}\|}; \qquad \lambda_{min}=\liminf_{n\to \infty} \inf_{m\ge 0} \sqrt[n]{\|dR^n|_{l_{g_m}}\|}.$$
\end{definition}
The definition of the expansion rate is close to the definition of the Lyapunov exponent. Note also that if $g_n$ is a $q$-periodic orbit, then both expansion rates coincide with the root of order $q$ of the modulus of the eigenvalue of $R^q$ that corresponds to the eigenvector contained in $l_{g_0}$.

Let $f_\mu$ be an analytic family of analytic self-maps of $\bbC/\bbZ$ defined  in a neighborhood of $\bbR/\bbZ$ for $\mu_j\in (\mathbb C^n, 0)$.  Suppose that for real $\mu_j$ and $\mu_1<0$, $f_\mu$ is a diffeomorphism of $\bbR/\bbZ$. Suppose that   for $\mu_1=0$, the maps $f_\mu$ are cubic critical circle maps with $0$ as a critical point. Suppose that $f'_{\mu}(0)$ has a zero of order  1 at $\mu_1=0$: when $\mu_1$ encircles zero, $f'_\mu(0)$ makes one turn around zero.
Finally, assume that $\frac {\partial f}{\partial \mu_2}>0$ on the circle whenever $\mu_1=0$.
An example to keep in mind is
the Arnold family, normalized as
\begin{equation}
  \label{eq:arnoldfamily}
  f_{\mu_1, \mu_2} (z)= z+\mu_2 + \left(\mu_1+\frac {1}{2\pi}\right) \sin 2\pi z.
  \end{equation}

\begin{theorem}
\label{th-smoothness}
Let $g\in \mathcal I$ be a cubic critical circle map with $\rot g=\alpha$ where $\alpha$ is of a bounded type. Let $\hat g \in p^{-1}(g)$  belong to $\Lambda$. Suppose that the renormalization operator  $\mathcal R\colon D\to D$ has an invariant direction field $\xi_{\mathcal R^n \hat g}\subset E^u_{\mathcal R^n \hat g}$, such that   $\xi_{\mathcal R^n \hat g}$ is uniformly transversal to $D_0$. Let $\lambda_1$ be the minimal expansion rate along the direction field $E^u_{\mathcal R^k \hat g} \cap D_0$, and let $\lambda_2$ be the maximal expansion rate along $\xi_{\mathcal R^k \hat g}$. Assume $\lambda_1>\lambda_2$.

Then in any $C^\omega$-family $f_{\mu}$,  $\mu = (\mu_1,\dots, \mu_n)\in \bbR^n$,  as above with $\rot f_{0} = \alpha$, the Arnold $\alpha$-tongue  is given by a function $\mu_2=\mu_2(\mu_1, \mu_3, \dots, \mu_n)$ that is at least $C^k$-smooth at $0$, where $k=\lfloor \frac{\log\lambda_1}{\log \lambda_2} \rfloor$.
\end{theorem}
Since $\mathcal R$ contracts on fibers of $p$ that project to $D_0$, we will see that $\lambda_1,\lambda_2$ only depend on $\alpha$ and not on $g$, $\hat g$. 

\begin{remark}
 The Arnold tongues for $\mu_1<0$ are analytic curves due to  Risler's theorem \cite{Risler,GY}. This result shows that they are at least finitely smooth at $\mu_1=0$. It is an open question whether these curves are analytic up to $\mu_1=0$.
\end{remark}

\begin{remark}
\label{rem-transversality}
 The condition $df/d\mu_2>0$ can be replaced by the condition that the vector field $df/d\mu_2$ is transversal to the surface $\{rot f = \alpha\}$ for $\mu=0$, see Sec. \ref{sec-stableunstable}.
\end{remark}

\begin{remark}
 For a rotation number $\alpha$ with a periodic continued fraction,  the assumptions of this theorem are equivalent to the condition that for some $g\in \Lambda$ with $\rot g = \alpha$, $d \mathcal R^q|_{\hat g}$ has two distinct  unstable eigenvalues, and the one with the larger absolute value corresponds to $E^u_{\hat g}\cap D_0$ (that is, to critical circle maps).
 
 Remark \ref{rem-periodic} shows that this assumption always holds true; however, for some periodic orbits,  $1< \frac{\log \lambda_1}{\log\lambda_2} <2$ and the conclusion of the theorem does not guarantee  smoothness of Arnold tongues higher than $C^1$.  
\end{remark}

In hyperbolic dynamics, it is standard to expect that if a hyperbolic map has a gap $(\lambda_2, \lambda_1)\subset \bbR$ in its spectrum, then the corresponding invariant manifolds are smooth, and the degree of smoothness is at least $\lfloor \log \lambda_1/\log \lambda_2 \rfloor$. This, together with Theorem~\ref{th-operator}, would imply Theorem~\ref{th-smoothness}, since the codimension-1 manifolds $\{\rot f = \text{const}\}$ form a ``weakly stable'' invariant foliation that corresponds to $\lambda_2$ and all stable eigenvalues of $\mathcal R$, and the corresponding spectral gap is $(\lambda_2, \lambda_1)$. In a Banach space setting, such statements were proven in \cite{ElBialy}.
However, \cite{ElBialy} requires either global estimates on the nonlinearity of the operator or the existence of smooth cut-off functions (that are scarce in Banach spaces), so we cannot refer to these results.
We  prove the smoothness directly in Theorem \ref{th-Banach-op}, see \S~\ref{sec-stableunstable}. We do not need to prove the existence of the invariant  manifold; we study the smoothness of the existing invariant manifold. As a result, we only use hyperbolicity of the operator, with no further assumptions.


The inequality for expansion rates required in Theorem \ref{th-smoothness} holds for the golden ratio rotation number. Indeed, the two top unstable eigenvalues of $\mathcal R$ at the period-two periodic orbit that corresponds to the critical map $f$ with the golden ratio rotation number were estimated rigorously numerically in \cite[Theorem 3.2 and p.19]{Mest}. They equal $\gamma_1\approx -2.83$ (corresponding to critical maps) and $\gamma_2\approx 1.66$.
Since $\frac{\log 2.83}{\log 1.66} = 2.05$, this agrees with the numerical results of \cite{LlaveLuque} on $C^2$-smoothness of Arnold tongues that correspond to the golden ratio rotation number.  

\begin{figure}[h]
  \caption{\label{fig-foliation}A schematic illustration of the hyperbolic picture for renormalization. The attractor in the space of diffeomorphisms is rigid rotations (left); the stable foliation has codimension $1$. The golden-mean leaf is indicated. It extends to a fixed point in the space of critical circle maps (right); the smoothness of the leaf at the critical boundary is determined by the spectral gap.}
\centerline{\includegraphics[width=0.5\textwidth]{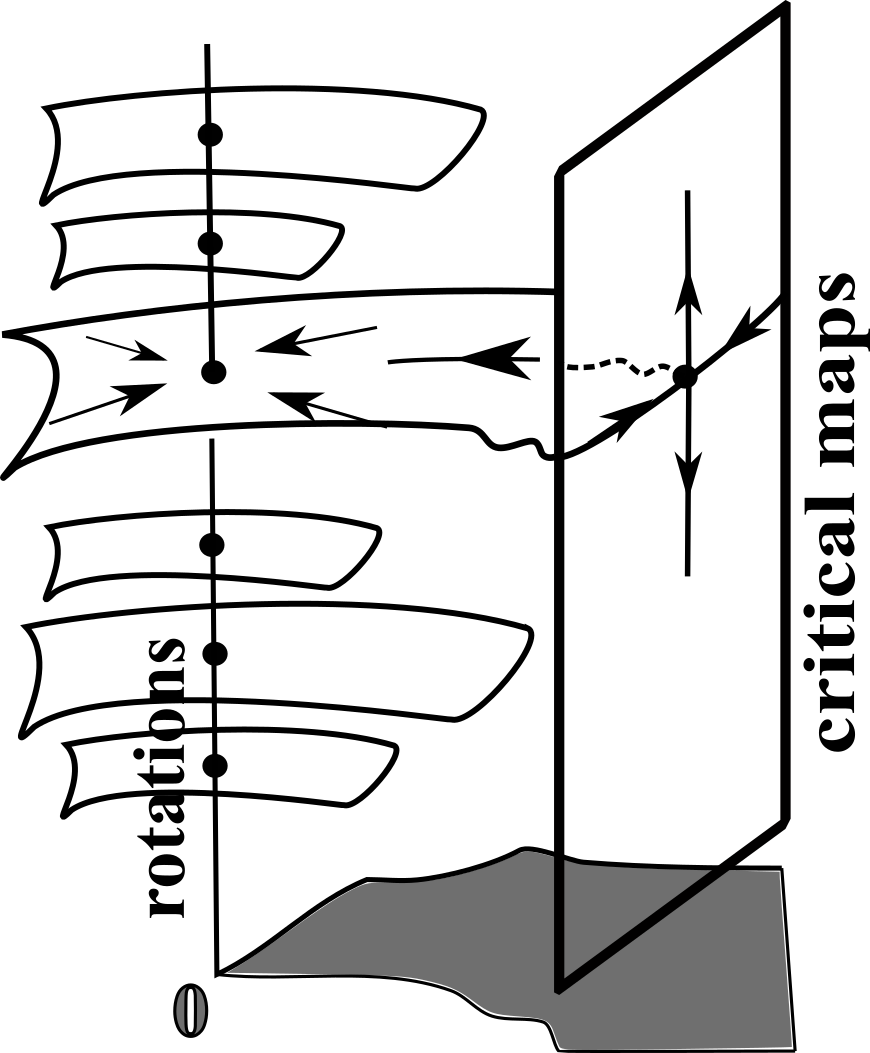}}
\end{figure}

Also, this inequality holds for orbits in $\Lambda$ that correspond to rotation numbers with a sufficiently high type, due to the next result. Thus the inequality holds for the two cases that are believed to be most extreme: high-type rotation numbers (with large denominators of continued fraction convergents), and the golden ratio (with the smallest possible denominators of continued fraction convergents).  It is an open question whether the inequality holds for all Herman numbers.

Consider the Gauss map $G(x)=\{\frac 1 x\}$ and set
$$\alpha_{n}=G^n(\alpha), a_n=[1/\alpha_{n}]$$
when defined.
Numbers $a_n$ are the coefficients of the continued fractional expansion of $\alpha$, and we have $$\alpha = \cfrac{1}{a_0 + \cfrac{1}{a_1+\dots}},$$
which we will abbreviate as 
$$\alpha=[a_0,a_1,\ldots].$$
We also consider the continued fraction convergents
$$\frac{p_n}{q_n}=[a_0,\ldots,a_{n-1}]$$
and write $\alpha \sim \{p_n/q_n\}$ for the correspondence between the number $\alpha$ and its continued fractional convergents.
\begin{proposition}
\label{prop-exponents}
For a real number $\alpha = [a_1,a_2,\dots]$ with $a_k>n$, let $g\in \mathcal I$ have rotation number $\alpha$ and let $\hat g \in p^{-1}g$  belong to $\Lambda$.

Then the minimal expansion rate $\lambda_1$ of the renormalization operator $\mathcal R\colon D\to D$ along the distribution $E^u_{\hat g}\cap D_0$ tends to infinity as $n\to \infty$. Also, $\mathcal R$ has an invariant line field in $E^u_{R^n (\hat g)}$ transversal to $D_0$, and the corresponding maximal expansion rate $\lambda_2 $ remains bounded as $n\to \infty$.
\end{proposition}

This, together with the above, immediately implies the corollary.

\begin{corollary}
\label{cor-smoothness-htype}
 Let $f_\mu$, $\mu = (\mu_1,\dots, \mu_n)\in \bbR^n$ be a family of analytic circle maps as above. Suppose that $\alpha=[a_1,a_2,\dots]$ with $n<a_k<m$. For sufficiently large $n$, the Arnold $\alpha$-tongue in the family $f_\mu$ is (at least) $C^k$ smooth, where $k=k(n)$ tends to infinity as $n\to \infty$.
\end{corollary}

 One can extend this result to rotation numbers such that in the corresponding  sequence $a_k$, the numbers with $a_k<n $ occur sufficiently rarely.

\begin{theorem}
\label{th-smoothness-htype2}
Let $f_\mu$, $\mu = (\mu_1,\dots, \mu_n)\in \bbR^n$ be a family of analytic circle maps as above. For any $\delta\ge 0$, for sufficiently large $K$,   if  $\alpha=[a_1,a_2,\dots]$ is of bounded type and for all $m$, for sufficiently large $n$, we have $$\frac{\#\{k \mid a_k<K, m\le k<m+n\}}{n}\le \delta,$$ then the Arnold $\alpha$-tongue in the family $f_\mu$ is (at least) $C^k$ smooth, where $k=k(K, \delta)$ tends to infinity as $K\to \infty$.
\end{theorem}

Let us conclude the introduction with a brief guide to the layout of the paper. In \S~\ref{sec-cylren} we recall the basic setting of cylinder renormalization for critical circle maps. In \S~\ref{sec-measure} we construct (-1)-measures for critical circle maps and prove an extension of Slammert's result on smoothness of open Arnold tongues to the endpoints. In \S~\ref{sec-triples} we describe a Banach manifold setting in which renormalization horseshoes of diffeomorphisms and critical circle maps can be combined, and prove Theorem~\ref{th-operator}. In \S~\ref{sec-derivatives} we give geometric estimates for the stable foliation of renormalization away from critical circle maps. Section~\ref{sec-stableunstable} contains the statement of a general result on the smoothness of stable-unstable manifolds in a Banach space setting (Theorem \ref{th-Banach-op}) and the proof of Theorem \ref{th-smoothness} modulo Theorem \ref{th-Banach-op}. In \S~\ref{sec-Banach-op} we prove Theorem~\ref{th-Banach-op}. Finally, \S~\ref{sec-unifhyp} contains a new construction of the expanding direction of renormalization of critical circle maps, with explicit estimates on the expansion rate as needed to derive Proposition~\ref{prop-exponents} and Theorem~\ref{th-smoothness-htype2}.

\section{Cylinder renormalization.}
\label{sec-cylren}
\subsection{Cylinder renormalization for cubic critical circle maps}
\label{sec-crit}
We are going to briefly recall the cylinder renormalization operator 
$\cren$ on $\mathcal D^{cr}_\eps$ constructed in \cite{Ya3} by the second author.
As before, $p_n/q_n$ will stand for the continued fractional convergents of $\rot f$.

\subsection*{Construction of the cylinder renormalization}


In \cite{Ya3}, \cite{Ya4}, the second author showed that  for any $f\in \dcrepsr$ there exists $N(f)$ such that for all $n\geq N(f)$, the iterate $f^{q_n}$ has two complex conjugate repelling fixed points joined by a crescent-like fundamental domain $C_f$ for the iterate $f^{q_n}$. The iterate $f^{q_n}$ glues the boundary curves of the crescent $C_f$ (minus the two repelling points at the ends) into a Riemann surface whose conformal type is $\CC/\ZZ$. 
The number $N(f)$ can be chosen uniformly on bounded subsets of $\dcrepsr$. Furthermore, if a critical circle map $f$ has a fundamental crescent $C_f$ for an iterate $f^{q_n}$, as above, then the same is true for non circle-preserving analytic maps in an open neighborhood of $f$ in  $\dcreps$.

Let us fix a bounded open set $\cU^\RR\subset\dcrepsr$ on which  $N=N(f)$ can be chosen uniformly. 
The fundamental crescent $C_f$ for the iterate $f^{q_N}$ can be chosen to move holomorphically with $f$ in an open neighborhood $\cU\supset \cU^\RR$ in the space $\dcreps$ of non circle-preserving maps, and to contain $0$ in its interior. There is a unique real-symmetric uniformizing coordinate $\Psi_f$ which sends the quotient  $C_f/f^{q_N}$ to $\CC/\ZZ$ with $\Psi_f(0)=0$. 
The map $\Psi_f$  depends analytically on $f\in \mathcal U$, and conjugates $f^{q_N}$ to the shift $z\to z+1$ (for odd $n$) and to the shift $z\to z-1$ (for even $n$)  in $C \cup f^{q_N}(C)$. Note that when $f\in\cU^\RR$ and the $N$-th digit of the continued fraction of $f$ is sufficiently large, the map $f^{q_N}$ is near-parabolic, and we can think of $\Psi_f$ as the perturbed Fatou coordinate of $f^{q_N}$ (see e.g. \cite{Sh}).

Let $P$ be the first-return map to $C$ under the action of $f$, defined in a neighborhood of $C\cap \bbR/\bbZ$.
By definition,
$$\cren f := \Psi_f P \Psi_f^{-1}.$$
%


\noindent
 We are ready to formulate \cite[Theorem 3.7]{Ya4}.
\begin{theorem}[\cite{Ya4}]
  \label{th-horseshoe}
  There exists a choice of an open bounded set $\cU\subset \dcreps$ as above and a corresponding choice of $N$ such that the following holds.
  With these choices, $\cren$ is an analytic operator $\cU\to \dcreps$, which is real-symmetric, and preserves $\dcrepsr$.  This operator has a hyperbolic invariant set $\cI\subset \dcrepsr\cap \cU$ with one-dimensional unstable direction.
  The set $\cI$ is pre-compact in the sense of uniform convergence. There is a map
$\iota\colon  {\mathcal I}\to \bbN^{\bbZ}$   that takes any $f\in {\mathcal I}$ to a bi-infinite sequence such that the positive part of this sequence coincides with the continued fractional expansion of $\rot f$, and $\iota$ conjugates $\cren|_{{\mathcal I}}$ to the $N$-th power of the Bernoulli shift.
\end{theorem}
The following theorem describes the stable foliation of this hyperbolic set.
\begin{theorem}{\cite[Theorem 3.8]{Ya4}}
  \label{th-convergence}
  If $f,g\in\dcrepsr$ and  $\rot f=\rot g$ is irrational, then
  $\dist( \cren^n f, \cren^n g) \to 0$
  at an eventually uniform geometric rate. Furthermore, all limit points of the sequence $\{\cren^n f\}$ lie in the closure of $\cI$.
\end{theorem}
An analogue of this result for the formalism of critical commuting pairs was earlier established by de Faria and de Melo \cite{dFdM2} under the assumption that the rotation number is of bounded type.
We will also use the following estimate. Here and below we assume that the crescent $f^{q_{N}}(C)$ contains $0$, and $\Psi_f(0)=0$.
 \begin{lemma}
 \label{lem-Psi-estim}
 For any $\eps, K$ and any sufficiently large $C$, one can choose $\cU$ and  $N$ as above so that
 for any $f\in \cU$,   we have  $\Psi_f'(0)>K$ and $\Pi_{K\eps} \subset \Psi_f (\Pi_\eps) $. 
 
 Also, for any $\delta>0$,  for sufficiently large $C$, the distortion of $\Psi_f$ is bounded by $C$ on the curvilinear rectangle  $\Psi^{-1}_f([a,a+1]\times[-i\eps, i\eps] )$ with $-1+\delta<a<\delta $: $$\frac 1C<\left|\frac{\Psi_f'(k)}{\Psi'_f(l)}\right|<C\text{ for }k, l\in [a, a+1]\times[-i\eps, i\eps]$$ and $$\left|\frac{\Psi''(k)}{\Psi'(k)}\right| < C\Psi'(0).$$
 \end{lemma}
The proof is by a reference to the Koebe Distortion Theorem and an upper estimate on the length of the interval $\Psi^{-1}[0,1]$, see \cite[p.20]{Ya4}.


\subsection{Real {\it a priori} bounds}

Let $I_n = [ f^{q_n}(0), 0]$. Recall that the dynamical partition $\mathcal P_n$ of the map $f$ is formed by the intervals $$f^l(I_n), 0\le l<q_{n+1}\text{, and }f^l(I_{n+1}), 0\le l < q_n.$$ The intervals of $\cP_n$ cover the circle and have disjoint interiors.
For further reference, let us formulate a statement of real {\it a priori} bounds for smooth critical circle maps; see \cite{dFdM1} for the proofs.
 \begin{theorem}[Theorem 3.1 \cite{dFdM1} Parts  (a), (e)]
\label{th-realbounds}
There exists a universal constant $C_0>1$ such that for every $C^3$-smooth critical circle map $f$ with a single critical point $0$, there exists $N_1=N_1(f)$ such that for any $n\geq N_1$ the following holds.
\begin{enumerate}
\item[(a)]
  Any two adjacent intervals $I,J$ of the partition $\mathcal P_n$ are $C_0$-commensurable:
  $$1/{C_0}<|I|/|J|< C_0;$$

\item [(b)] For any $0<i\le j \le q_{n+1}$, the distortion of the restriction of $f^{j-i}$ to $f^{i}(I_n)$ is bounded by $C_0$.
\end{enumerate}
The number $N_1$ can be chosen uniformly on a $C^2$-precompact set of critical circle maps with irrational rotation numbers.
\end{theorem}

Real {\it a priori} bounds  imply the following geometric estimate on $|I_n|$:
\begin{theorem}[Theorem 3.1  \cite{dFdM1} Part (b)]
  \label{th-intervals}
  There exists $\delta\in(0,1)$ such that for any $C^3$-smooth critical circle map $f$ with an irrational rotation number $\rho$ there exists $N_2=N_2(f)$ such that for all $n\geq N_2$ all 
  intervals of the partition $\mathcal P_n$ are shorter than $\delta^n$.

  The number $N_2$ can be chosen uniformly on a $C^2$-precompact set of critical circle maps with irrational rotation numbers.
\end{theorem}

Let us note that since maps in $\cI$ are infinitely many times backward renormalizable, for all of them one can set $N_1=N_2=1$ in the above statements.

\section{$C^1$-smoothness of irrational Arnold tongues}
\label{sec-measure}
In this section, we prove that all Arnold tongues that correspond to irrational rotation numbers are $C^1$-smooth at critical circle maps. This result does not require an inequality on eigenvalues from Theorem \ref{th-smoothness} and applies to all irrational rotation  numbers.

The proof is inspired by, and partly based on, the results of L. Slammert, see \cite{Slammert}.
Recall that $\depsr$ is the Banach manifold of $\bbR/\bbZ$-preserving analytic maps defined in a strip of wdth $\eps$.
The Banach submanifold $\dcrepsr$ of $\depsr$ consists of cubic critical circle maps with $f'(0)=f''(0)=0$, and its tangent space is formed by analytic vector fields in the strip of width $\eps$ with $v'(0)=v''(0)=0$. 

\begin{definition}
 The (-1)-measure of a smooth circle map $f$ is any measure that satisfies
 $$\int \phi(x)d\mu = \int f'(f^{-1}(x))\phi(f^{-1}(x)) d\mu$$ for any continuous test function $\phi$.
\end{definition}
Note, that for circle diffemorphisms this is equivalent to the definition of a (-1)-measure used by Slammert in \cite{Slammert}:
$$
\int \xi (f(x)) \frac 1{f'(x)} d\mu = \int \xi d\mu,
$$
if we use test functions of the form  $\xi(x) = \phi(f^{-1}(x))f'(f^{-1}(x))$.
Slammert attributes the following result to Douady and Yoccoz (see \cite{dMPugh}):
\begin{theorem}
\label{th-1meas}
 A circle diffeomorphism with an irrational rotation number has a unique (-1)-measure.
\end{theorem}
This measure might be singular for Liouville rotation numbers. If  $f$ is conjugate to a rotation via a circle diffeomorphism $h$ (for instance, if the rotation number of $f$ is Herman), then it is easy to check that its (-1)-measure $\mu_f$ is absolutely continuous with density $(h'(x))^2$.

We can synthesize the following from \cite[Theorem 2.8]{Slammert}:
\begin{theorem}
  Let $g$ be a circle diffeomorphism with an irrational rotation number $\alpha$, and let $\mu$ be the (-1)-measure of $g$. Then the 
    condition $\int v|_{g^{-1}(x)} d\mu =0$ defines the tangent space to $\{\rot f=\alpha\}$ at $g$.
  \end{theorem}

Note, that
if $\alpha$ is a Herman number (see \cite{Yoccoz2002} for the definition of the class of Herman numbers $\cH$),
we can prove this directly without referring to \cite{Slammert}.
Recall that all circle diffeomorphisms with Herman rotation numbers are analytically conjugate to  rotations \cite{Yoccoz2002}; as shown in \cite{GY}, the conjugacy depends analytically on the diffeomorphism in the Banach submanifold $\{\rot f=\alpha\}$. Let $h$ be the conjugacy, $g(x) = h^{-1}(h(x)+\alpha)$. Consider the probability measure $\mu$ with density proportional to  $(h'(x))^2$.  Then $\mu$ satisfies
\begin{equation}
\label{eq-1}
 \int \phi(x) d\mu = \int (g)'(g^{-1}(x)) \phi(g^{-1}(x)) d\mu.
\end{equation}
 The tangent space to $\{\rot f=\alpha\}$ at $g$ consists of vector fields of the form $$v = \frac{d}{d\xi} (id+\eps w) g(id + \eps w)^{-1} = w|_{g(x)} - g'(x) w$$ where $w$ is analytic. For any such vector field, we have
$$\int v|_{g^{-1}(x)} d\mu= \int w - g'(g^{-1}(x)) w(g^{-1}(x)) d\mu = \int w d\mu - \int w d\mu =0,$$
the second identity is due to \eqref{eq-1}. Thus the codimension-1 condition $$\int v|_{g^{-1}(x)} d\mu =0$$ defines the tangent space to $\{\rot f=\alpha\}$ at $g$.

\smallskip

We will prove the following two theorems:
\begin{theorem}
\label{th-measure}
 A cubic critical circle map $f$ with an irrational  rotation number has a unique (-1)-measure $\mu_f$. The tangent space to the  local analytic manifold $\{\rot g = \alpha, g\in \dcrepsr\}$, at the critical circle map $f$ coincides with $\{v\in T \dcrepsr\mid \int v|_{f^{-1}(x)} d\mu_f=0\}$.

 If $g_n\to g$ are critical circle maps with irrational rotation numbers, then the corresponding (-1)-measures weakly converge.
\end{theorem}

\begin{theorem}
\label{th-C1}
 For any sequence $f_n\to f$ where $f_n\in \depsr$ are circle diffeomorphisms and $f$ is a cubic critical circle map with $f'(0)=f''(0)=0$, if $\rot f_n=\rot f = \alpha\in \bbR\setminus \bbQ$, then (-1)-measures $\mu_n$ of $f_n$ converge to the (-1)-measure $\mu_f$ of $f$. Consequently, the tangent spaces to the local analytic manifolds $\{\rot g = \alpha\}$ at $f_n$ converge to  the space given by  $ \int v|_{f^{-1}(x)} d\mu_f =0$ which intersects $T\dcrepsr$ on the tangent space to the local analytic manifold $\{\rot g = \alpha, g\in \dcrepsr\}$, at $f$.
\end{theorem}

\noindent
The convergence of the tangent spaces has an immediate corollary:
\begin{corollary}
  \label{cor:C1}

Consider a $C^\omega$-smooth family $f_{\mu}$, $\mu\in \bbR^n$, defined on a semi-neigh\-bor\-hood $\mu_1\in (-\eps, 0]$, $\mu_k \in (-\eps,\eps)$, such that $f_{\mu}$ is an analytic circle diffeomorphism for $\mu_1<0$ and a cubic critical circle map for $\mu_1=0$. Assume that $\frac {\partial f}{\partial \mu_2}>0$ on the circle whenever $\mu_1=0$. Then for each irrational rotation number $\alpha$, the Arnold $\alpha$-tongue  is given by a function $\mu_2=\mu_2(\mu_1, \mu_3, \dots, \mu_n)$ that is at least $C^1$-smooth at $0$.
  \end{corollary}

\begin{proof}[Simultaneous proof of Theorem \ref{th-measure} and Theorem \ref{th-C1}]
$\-$ \\
   \textbf{Step 1: Partial limits of tangent spaces are given by (-1)-measures of critical circle maps.}

   In assumptions of Theorem \ref{th-C1}, choose a weakly convergent subsequence $\mu_{n_k}$ from $\mu_n$. Let $\mu$ be the limit measure.

Due to the weak convergence of measures, \eqref{eq-1} implies that
\begin{equation}
\label{eq-1-crit}
 \int \phi(x) d\mu = \int f'(f^{-1}(x)) \phi(f^{-1}(x)) d\mu
\end{equation}  for any continuous test function $\phi$. Indeed, the left-hand sides of \eqref{eq-1} converge to the left-hand side of \eqref{eq-1-crit}. In the right-hand side, the difference between the integrals of $(f_{n_k})'(f_{n_k}^{-1}(x))\phi(f_{n_k}^{-1}(x)) $ with respect to  $\mu_{n_k}$ and $\mu$ is small due to the weak convergence, and for any fixed  continuous test function, the distance between
$(f_{n_k})'(f_{n_k}^{-1}(x))\phi(f_{n_k}^{-1}(x))$ and $f'(f^{-1}(x)) \phi(f^{-1}(z)) $ is small for large ${n_k}$.

We conclude that as $k\to \infty$, the tangent spaces to $\{\rot g = \alpha\}$ at $f_{n_k}$ have a limit, namely a codimension-1 space of vector fields with  $\int v|_{f^{-1}(x)} d\mu =0$, where $\mu$ is a (-1)-measure for $f$.

\noindent
\textbf{Step 2: relation of (-1)-measures for critical circle maps to tangent spaces of $\{\rot f=\alpha\}$. }

We now prove that this limit space intersects the space of critical circle maps on the tangent space of the condition $\{\rot f=\alpha\}$.
By real {\it a priori} bounds,  all critical circle maps with $\{\rot g=\alpha \}$ are quasiconformally conjugate to $f$.
We are now going to use a stronger version of this statement, derived from the fact that the set $\{\rot g=\alpha\}\cap \dcrepsr$ lies inside a complex-analytic Banach manifold $W\subset\dcreps$, which is a stable manifold of the renormalization operator constructed in \cite{Ya4}. As shown in \cite{Ya4}, the maps in $W$ are quasiconformally conjugate in a neighborhood of $\RR/\ZZ$, with the conjugacy varying analytically on $W$. 
%
The tangent space to $\{\rot g=\alpha\}\cap \dcrepsr$, which is the real slice of $W$,  consists of the vector fields of the form $v =  w|_{f(x)} - f'(x) w$ where $w$ is
a real-symmetric quasiconformal vector field varying analytically with $g$.
Hence,
$$\int v|_{f^{-1}(x)} d\mu =  \int w - f'(f^{-1}(x)) w(f^{-1}(x)) d\mu = \int w d\mu - \int w d\mu =0.$$
Thus, the tangent space to $\{\rot g = \alpha\}\cap \dcrepsr$ lies in the linear subspace given by the condition $\int v|_{f^{-1}(x)} d\mu = 0$, and, by the dimension count, coincides with it.

\noindent
\textbf{Step 3: uniqueness of (-1)-measures for critical circle maps.}

Let us prove the first statement of Theorem \ref{th-measure}:  a critical circle map  $f$ has a unique (-1)-measure. This  also implies that the limit $\mu$ does not depend on the sequence $f_{n_k}$ and thus completes the proof of Theorem \ref{th-C1}.

Suppose that we have two  different (-1)-measures $\mu$ and $\tilde \mu$. Combine them with such coefficients $c_1,\;c_2$ that  $$\int v|_{f^{-1}(x)} d(c_1 \mu + c_2 \tilde \mu) =0$$ for all vector fields in $T_f\dcrepsr$. For the codimension $1$ tangent subspace  to $\{\rot g = \alpha\}$ this is automatic. For the remaining direction, this can be achieved by an appropriate choice of the constants.
The following lemma applies to the charge $(c_1 \mu + c_2 \tilde \mu)$ and thus implies that $\mu = \tilde \mu$.

\begin{lemma}
 Any charge $\nu$ with $\int v(f^{-1}(z)) d\nu=0$ for any vector field $v\in T_f\dcrepsr$ is zero.
\end{lemma}
\begin{proof}
 For any vector field $v$, choose large $n$ and let $$\tilde v(z) = v(z) + a_1 \sin n z + a_2 \cos n z$$ where $a_1, a_2$ are chosen so that $\tilde v'(0)=\tilde v''(0)=0$.
 Then $\tilde v \in T_f\dcrepsr$, and $ a_1, a_2\to 0$ as $n\to \infty$. We have $$\left|\int v(f^{-1}(z)) d\nu\right| =$$ $$=\left|\int \tilde v(f^{-1}(z)) d\nu - a_1 \int \sin n(f^{-1}(z)) d\nu - a_2 \int \cos n(f^{-1}(z)) d\nu\right|\le$$ $$ \le c|a_1| + c|a_2| \to 0 \text{ where }c=\int 1 d\nu <\infty.$$ This implies $\int v(f^{-1}(z)) d\nu =0$ for any $v$, thus $\nu=0$.
\end{proof}

This completes the proof of  Theorem \ref{th-C1}.

\noindent
\textbf{Step 4: convergence of (-1)-measures for critical circle maps}

It remains to prove the first statement of Theorem \ref{th-measure}: if $g_n\to g$ are critical circle maps, the corresponding measures weakly converge. Indeed,  any weak limit of a sequence of (-1)-measures for $f_n$ is a (-1)-measure for $f$, and since such measure is unique, the statement follows.
\end{proof}

We conclude by proving the following lemma which will be useful to us in what follows, and is interesting in its own right.
\begin{lemma}
\label{lem-noatoms}
The (-1)-measure for a critical circle map $f$ with irrational rotation number cannot have atoms.
\end{lemma}
\begin{proof}
Fix $p$ and apply the definition of the (-1)-measure  for continuous test functions that approximate $\chi([p-\frac 1n, p+\frac 1n])$, then let $n\to \infty$. We get that $\mu_f(\{p\}) = \mu_f(\{f(p)\})\cdot f'(p)$. Iterating this relation, we get $\mu_f(\{p\}) = \mu_f(\{f^k(p)\})\cdot (f^k)'(p)$ for any $k>0$, $p\in \bbR/\bbZ$.

If $p$ is a preimage of a critical point, $f^k(p)=0$, this implies $\mu_f(\{p\}) = 0$.

Otherwise we get that
\begin{equation}
\label{eq-meas-series}
1\ge \mu_f ( \{p, f(p), f^2(p), \dots\}) = \mu_f(\{p\}) \cdot \left( 1+\sum_{k=1}^{\infty} \frac{1}{f^{(k)}(p)}\right).
                     \end{equation}
Let us prove that this series diverges.
Indeed, set $$A_n(p)=[p,f^{q_n}(p)],\text{ and } B_n(p)=[f^{-q_n}(p),f^{2q_n}(p)].$$
By real {\it a priori} bounds \ref{th-realbounds}, $f^{q_n}(B_n(p))$ contains a $C=C(f)$-scaled neighborhood of $f^{q_n}(A_n(p))$, and $f^{q_n}(A_n(p))$ is $C(f)$-commenurable with $A_n(p)$ (the value of $C$ becomes universal for large enough $n$). Furthermore, since $q_n$ is a first return iterate, the orbit $B_n(p)\mapsto\cdots\mapsto f^{q_n}(B_n(p))$ passes through each critical point of $f$ at most three times. Hence, by one-dimensional Koebe Distortion Theorem, the iterate $f^{q_n}|_{A_n(p)}$ decomposes into a universally bounded number of power maps and a universally bounded number of maps with $K(f)$-bounded distortion (again, $K$ becomes universal for large $n$). Hence, there exists $c=c(f)$ (also, universal for large $n$) such that $(f^{q_n}(p))'<c$.


Since this holds for any $n$, the series diverges, thus $\mu_f(\{p\})=0$.
\end{proof}

%
%
%

%

\section{Renormalization in the space of triples}
\label{sec-triples}
In this section, we define a Banach space setting for renormalization in which the results for homeomorphisms and critical circle maps can be combined, and prove Theorem~\ref{th-operator}.
\subsection{Space of triples}
Recall that $\Pi_\eps$ is a strip of width $\eps$ around $\bbR/\bbZ$ in $\bbC/\bbZ$, and let
$$\bar \Pi_\eps=[-0.5,0.5]\times[-i\eps, i\eps]\subset \bbC.$$
Set
\begin{equation}
  \label{eq-triples1}
  H_{a}(z) = z^3 - az.
  \end{equation}

We start with a helpful lemma:
\begin{lemma}
\label{lem-pi-exist}
Let $\eps>0$ and consider a closed set $\Omega^\bbR$ of analytic maps $G\colon \bar\Pi_\eps\to \bbC$ such that  $G(0)=0$,  $G$ is $\bbR$-preserving, $\max_{\bar \Pi_{\eps}}|G''|$ is uniformly bounded, $G'|_{[-0.5,0.5]}$ is uniformly bounded away from zero.

For sufficiently small $\delta$ and a small neighborhood $\Omega\supset \Omega^{\bbR}$, we can choose  $\tau_1, \tau_2>0$, the domains  $K = K_{a, G}\subset H_a G (\bar  \Pi_\eps)$, and  maps $\pi_{a,G} \colon K\to \bbC$ with the following properties.

\begin{itemize}
\item For any $G\in \Omega$,  for any  $|a|<\delta$, the map $\pi_{a, G}$ is a biholomorphism between $K$ and $\pi_{a, G}(K)$, continuous on $\overline  K$.  It depends analytically on $a,G$.

\item There exists $0<\tau_1<\eps, 0<\tau_2$ such that for any $G\in \Omega$, $|a|<\delta$, the map $\pi_{a,G}\circ H_{a} \circ G$ is defined in $\bar  \Pi_{\tau_1}$ and commutes with the unit translation. The projection of the domain $$\pi_{a,G}\circ H_a\circ G(\bar \Pi_{\tau_1}) \subset \bbC$$ to $\bbC/\bbZ$ contains $\Pi_{\tau_2}$.



\item  If $G\in \Omega^{\bbR}$ and $a\in \bbR$, the map $\pi_{a,G}$ is $\mathbb R$-preserving.


\item   If $H_{a} \circ G$ commutes with the unit translation, then $\pi_{a,G}=\id$.

\end{itemize}

\end{lemma}

\begin{proof}
Choose small $\nu$ and consider the image of the left side and the right side of $\bar \Pi_\nu$  under $H_a \circ G$. Join endpoints of these curves by segments. If $\nu=\nu(\Omega^{\bbR})>0$  is sufficiently small,  $a$ is sufficiently close to zero, and $G\in \Omega$ is sufficiently close to $\Omega^{\bbR}$, the bounds on $G', G''$ imply that  these curves bound a curvilinear rectangle $K$ with non-self-intersecting boundary and  $K\subset H_a \circ G(\bar  \Pi_\eps)$. 


The map $H_a\circ G ( G^{-1}\circ H^{-1}_a(z)+1)$  takes the left side of $K$ to its right side. Consider a smooth map $\Xi$ that depends analytically on $a, G$, takes $\bar \Pi_\nu$ to $K$ and conjugates  $z\to z+1$ to $H_a\circ G ( G^{-1}\circ H^{-1}_a(z)+1)$. This map  $\Xi$  is a  linear homotopy between the restrictions of the map $ H_a\circ G$ to the left and to the right boundary of $\bar \Pi_\nu$:
$$
\Xi (x+iy)= H_a G (-0.5+iy)\cdot  (0.5-x) + H_a G (0.5+iy) \cdot (0.5+x).
$$
Note that if $G$ preserves the real line and $a$ is real, then $\Xi$ preserves the real line.

Let us verify that the fraction  $\left| \Xi_{\bar z} / \Xi_z\right|$ is bounded away from 1:  we have
$$
\begin{aligned}
\Xi_x =& H_a G (0.5+ iy) - H_a G(-0.5+iy),   \\
-i\Xi_y =&  (H_a G)'(-0.5+iy) \cdot (0.5-x) + (H_a G)'(0.5+iy)\cdot  (0.5+x)
\end{aligned}
$$
and
$$
\frac {\Xi_{\bar z}} {\Xi_z} =  \frac{\Xi_x+i\Xi_y}{\Xi_x-i\Xi_y}
$$
Due to the bound on $G'$, the values $H_0G(0.5)$, $-H_0G(-0.5)$, $(H_0G)'(\pm 0.5)$  are positive, uniformly bounded and bounded away from zero for $G\in \Omega^{\bbR}$. Due to our bound on $G''$, we can choose small $\Omega\supset \Omega^{\bbR}$,  $\nu$, and $\delta$, so that  both $\Xi_x$ and $-i\Xi_y$ are close to $\bbR^+$, uniformly bounded and bounded away from zero for all $|a|<\delta, (x,y)\in \bar \Pi_\nu$. Thus the fraction  $\left| \Xi_{\bar z} / \Xi_z\right|$ is  uniformly  bounded away from 1.

Define a Beltrami differential $\mu$ on the plane by extending $ \Xi_{\bar z} / \Xi_z$ to the horizontal strip of width $\nu$ periodically, and by setting $\mu=0$ outside the strip. Let $\psi$ be a solution of the Beltrami equation $\psi_{\bar z}/\psi_z =\mu$ with the Beltrami differential $\mu$, normalized by $\psi(0)=0, \psi(1)=1, \psi(\infty)=\infty$; it exists due to the Ahlfors-Bers-Boyarski theorem \cite{Ahl-B}.
Set $\pi_{a, G} = \psi\circ \Xi^{-1}$. By construction, $\pi_{a,G}$ is an analytic univalent map in $K$ that extends continuously to its boundary and conjugates $$H_a\circ G ( G^{-1}\circ H^{-1}_a(z)+1)$$ to  the unit translation; thus  $\pi_{a,G}\circ H_{a} \circ G$ commutes with the unit translation on its domain. 
If $G\in \Omega^{\bbR}$, then $\Xi$ is $\mathbb R$-symmetric, thus the solution of the Beltrami equation is $\mathbb R$-symmetric. We conclude that in this case, $\pi_{a,G}$ preserves the real axis.

Let us prove that the map  $\pi_{a, G}$ depends analytically on $a,G$. Indeed, $\Xi $  (thus $\mu$) depends on $a,G$ analytically by construction, and $\psi$ depends on $a,G$ analytically due to the Ahlfors-Bers-Boyarski theorem. It is easy to see that an analytic compositional difference of two (non-analytic) maps depends analytically on a parameter if both maps depend analytically on a parameter.
 This implies the statement.

If $\delta$, $\Omega\supset\Omega^{\bbR}$ are sufficiently small, there exists $\tau_1>0$ such that $H_a G (\bar  \Pi_{\tau_1}) \subset K$ for $|a|<\delta,G\in \Omega$. Thus $\pi_{a,G} H_a G$ is well-defined in $\bar \Pi_{\tau_1}$.

{\it A posteriori}, we know that the solution of the Beltrami equation $$\psi = \pi_{a,G}\circ\Xi$$ is smooth, with non-degenerate differential in $\bar \Pi_{\nu}$. Since it depends analytically on $G$ in $\Omega$, the norm of its differential is bounded from below on $\bar \Pi_{\nu}$ for $G\in \Omega^{\bbR}$. The set $\Xi^{-1} (H_0G(\bar \Pi_{\tau_1}))\subset \bar \Pi_{\nu}$ contains a strip around $[-0.5, 0.5]$ of width uniformly bounded from below, thus the projection of $\pi_{0,G}(H_0G(\bar \Pi_{\tau_1})) = \psi(\Xi^{-1} (H_0G(\bar \Pi_{\tau_1}))$ to $\bbC/\bbZ$  contains a strip around $\bbR/\bbZ$ of width uniformly bounded from below. Hence we can fix $\tau_2>0$ so that for small $\delta$ and $\Omega\supset \Omega^{\bbR}$, for $|a|<\delta, G\in \Omega$, the projection of $\pi_{a, G}H_aG(\bar \Pi_{\tau_1})$ to $\bbC/\bbZ$ contains $\Pi_{\tau_2}$.



 If $H_a G$ commutes with the unit translation, then $H_a\circ G ( G^{-1}\circ H^{-1}_a(z)+1)$ coincides with the unit translation on the left side of the rectangle $K$, then $\Xi=\psi=\id$ and thus $\pi_{a,G}=\id$.

\end{proof}
%

\begin{definition}
  \label{def-triples}
The Banach manifold $D=D_{\eps, \eps'}$ in Theorem \ref{th-operator} is the space of triples of maps $(F, H_{a}, G)$ where
\begin{itemize}
 \item $F\colon \bbC/\bbZ\to \bbC/\bbZ$ is conformal on $\Pi_{\eps'} \subset \bbC/\bbZ$;
 \item $G\colon \bbC\to \bbC$ satisfies $G(0)=0, G'(0)=1$ and is  conformal in a neighborhood of $\bar \Pi_\eps$.
\end{itemize}
The values of $\eps, \eps'$ will be chosen later. With the uniform norm in the domains of $F$, $G$ specified above,  $D$ is an analytic Banach manifold. Let $D_0\subset D$ be given by $a=0$.

We will say that the triple is $\bbR$-preserving if its $F$, $G$-components preserve the real line and $a$ is real.
\end{definition}



Using Lemma~\ref{lem-pi-exist}, we define the projection $p\colon D\to \mathcal D_\tau$.
We set
$$p \colon \hat f = (F, H_{a}, G) \mapsto F \circ \pi_{a, G} \circ H_{a} \circ G.$$
We will use the same notation $p(\hat f)$ for an analytic extension of $p(\hat f)$ to any larger strips. Note that the projection $p$ is not defined on the whole space $D$; its domain depends on $\tau$. 


The next lemmas  show that we can lift critical circle  maps and analytic families of circle maps that approach critical circle  maps to the space of triples.
\begin{lemma}
\label{lem-lift-cr}
For any $\eps, \eps'>0$, there exists an analytic map $$j\colon f\to \hat f =(F, H_0, G),\; j\colon \mathcal D_{\eps}^{cr} \to D_{\eps, \eps'}$$ such that $p(j(f))=f$ on the domain of $j$, and $j$ is  defined for all critical circle maps $f\in \mathcal D_{\eps}^{cr, \bbR}$ that  have no preimages of the critical value $f(0)$ in $\Pi_\eps$ except zero.
\end{lemma}
\begin{proof}
 Set $F (z)= z+f(0)$, $G(z) = \sqrt[3]{f(z)-f(0)}$. For $G$, we choose the branch of the cubic root that preserves the real line for critical circle maps $f\in \mathcal D_{\eps}^{cr, \bbR}$, which defines the branch unambiguously in a neighborhood of $\mathcal D_{\eps}^{cr, \bbR}$ in $\mathcal D_{\eps}^{cr}$.    Set $j( f) = (F, H_0, G)$. Then $G$ is well-defined in a strip $\Pi_{\eps}$ since it contains no preimages of $f(0)$ under $f$ except zero, the  map $H_0\circ G=f(z)-f(0)$ commutes with the shift by $1$, and thus $\pi_{0, G}=\id$. We get $F\circ \pi_{0, G} \circ H_0 \circ G = f$.

 We have $G(0)=0$; to guarantee $G'(0)=1$, we will have to replace $G$ by $C\cdot G$ for an appropriate $C$. This replaces $\pi_{0, G}$ by $z\mapsto z/C^3 $ and does not affect the equality  $F\circ \pi_{0, G} \circ H_0 \circ G = f$.
\end{proof}

\begin{lemma}
\label{lem-lift}
 Let $f_\mu$ be an analytic family of self-maps of $\bbC/\bbZ$ defined for $\mu_j\in (\mathbb C^k, 0)$, $z\in \Pi_\eps$. Suppose that $f_\mu^{-1}(f_\mu(0))=0$ in $\Pi_{\eps}$.  Suppose that for real $\mu_j$ and $\mu_1<0$, $f_\mu$ is a diffeomorphism of $\bbR/\bbZ$. Suppose that   for $\mu_1=0$, the maps $f_\mu$ are cubic critical circle maps with $0$ as a critical point. Assume that $f'_{\mu}(0)$ has a zero of order  1 at $\mu_1=0$: when $\mu_1$ encircles zero, $f'_\mu(0)$ makes one turn around zero.

 Then for some $\delta$, for all $\mu$ with $ \mu_1\in(-\delta, 0)$ and $\dist (\mu, 0)<\delta$, we can define homeomorphisms  $h_\mu$ and triples $F_\mu, H_{a(\mu)}, G_\mu$ such that $F_\mu$ is holomorphic everywhere,  $G_\mu$ is a holomorphic map defined in $\bar  \Pi_{\eps}$, all these objects depend analytically on $\mu$, and $$h_\mu^{-1} f_\mu h_\mu=p(F_\mu, H_\mu, G_\mu).$$
 
 For $\mu_1=0$ and real $\mu_j$, the triple $(F_\mu, H_{a(\mu)}, G_\mu)$ coincides with $j(f)$ from the previous lemma, and $h_\mu=\id$.
\end{lemma}
\begin{proof}
Slightly abusing the notation, we will keep the notation $f_\mu$ for a lift $f_\mu\colon \bar \Pi_\eps\to \bbC$.

We need to satisfy $$h_\mu F_\mu \pi_{a(\mu),G_\mu} H_{a(\mu)} G_\mu h_\mu^{-1}(z)  = f_\mu(z)\text{ for }G_\mu(0)=0, G'_\mu(0)=1.$$ First, we will satisfy these equalities without the requirement $G_\mu'(0)=1$, and then we will make the final adjustments.

Since $f_\mu$ for $\mu_1=0$ is a cubic critical map, for small $\mu_1\neq 0$, $f_{\mu}$ has a pair of quadratic critical points $c_1, c_2$ close to zero, given by $(f_\mu)'(c_{1,2})=0$.
 When $\mu_1$ makes one turn around zero,  these critical points make half of a turn each and exchange; this follows from the fact that $(f_\mu)'(0)$ has a simple root at $\mu_1=0$. Namely, it is clear that the pair of critical points depends analytically on $\mu$ when $\mu_1\neq 0$, thus we only need to count the number of turns; since $$(c_1)'_{\mu_1} = -(f''_{z\mu_1}/f''_z)|_{c_1}  \sim \text{const}/ c_1,$$ we have $c_1\sim \sqrt{\mu_1}$ and this number is 1/2.

 Let  images of $c_1, c_2$ under $f$ be $d_1, d_2$. Then $d_1$, $d_2$ make $3/2$ of the turn each around $f(0)$ when $\mu_1$ makes 1 turn around zero, and they also exchange, since $$f_\mu(c_1)-f_\mu(0) \sim f'_z(c_1) c_1 \sim \mu_1^{3/2}.$$  Also, they are complex conjugate for real $\mu$, $\mu_1<0$.

 Both maps $F_\mu, h_\mu$ will be translations. We set $$F_\mu h_\mu (z)=z+(d_1+d_2)/2;$$ note that this composition depends analytically on $\mu$ (the translations $F_\mu, h_\mu$ themselves will be chosen later). Then the critical values of the map $(F_\mu h_\mu)^{-1}f_\mu$ are symmetric with respect to $0$ and purely imaginary for real $\mu, \mu_1<0$. Also, they make $3/2$  of the turn each around $0$ when $\mu_1$ makes 1 turn around zero. The difference between these points is still $d_1-d_2$.

  Set $$A(\mu)= \left(\frac{3\sqrt{3}(d_1-d_2)}{4}\right)^{2/3};$$ we choose the branch of the cubic root that is real for $\mu_1\in \bbR_{-}$ (that is, when $d_1- d_2$ is purely imaginary).  Then the critical values of $H_A$, namely $\pm \frac{2A}{3}\sqrt{A/3}$, coincide with the critical values $\pm \frac{(d_1-d_2)}{2}$ of  $(F_\mu h_\mu)^{-1}f_\mu$. Since $d_1-d_2$ makes $3/2$ turns around the origin as $\mu_1$ makes one turn around the origin, $A$ depends analytically on $\mu$ and has a simple zero at $\mu_1=0$.

Consider the 3-valued map  $H_A^{-1}(F_\mu h_\mu)^{-1}f_\mu$. It is defined in $\bar \Pi_{\eps}$.  For each $\mu$, it does not branch at critical points of $f_\mu$ since critical values of $H_A$ and $(F_\mu h_\mu)^{-1}f_\mu$ coincide. It might branch at other preimages of critical values of $f_\mu$, but our assumptions imply that the strip $\Pi_{\eps}$ does not contain such points.  Also, it is bounded on $\bar \Pi_\eps$. Thus this formula for each $\mu$ defines three distinct holomorphic functions in $\bar\Pi_\eps$ that depend analytically on $\mu$. We choose the one that is real for real $\mu$, $\mu_1<0$, and let it be equal to $G_\mu h_\mu$. Finally, we choose the translation  $h_\mu$ so that $G_\mu (0)=0$. Since $F_\mu h_\mu$ was chosen above, this uniquely defines $F_\mu$.

Note that for $\mu_1=0$,  we get $$A=0,  \; F_\mu h_\mu (z)=z+f_{\mu}(0),\; G_{\mu}(z) = H_0^{-1}(f_\mu(z)-f_\mu(0))\text{, and }h_\mu=\id,$$
so for critical maps in the family $f_\mu$, the lift coincides with $j(f_\mu)$ from Lemma \ref{lem-lift-cr}.

We now have  $h_\mu F_\mu H_{A(\mu)} G_\mu h_\mu^{-1}(z)  = f_\mu(z)$, all maps depend analytically on $\mu$, and $G_{\mu}(0)=0$.  
Since $H_{A(\mu)}\circ G_\mu$ is a circle map, $\pi_{A(\mu),G_\mu}=\id$ (see the last statement of Lemma \ref{lem-pi-exist}).

It remains to adjust our triple to achieve $G_\mu'(0)=1$.
Note that  $G_\mu'(0)\neq 0$ for $\mu_1=0$ since $f$ has a cubic critical point and $H_0^{-1}(z)=\sqrt[3]{z}$, so $G_\mu'(0)=C_\mu$ is a nonzero analytic function for small $\mu$.
Since $$H_A \circ Cz = C^3 H_{A/C^2}z,$$ we can make this adjustment while preserving analytic dependence of $A, F, G$ on $\mu$: we replace $G_\mu$ by $\frac{1}{C_\mu}G_\mu$, $A$ by $A/C_\mu^2$, $F_\mu$ does not change, and $\pi_{a,G} $ becomes $z\mapsto z\cdot C_\mu^3$.  This completes the proof.
\end{proof}

\subsection{Renormalization for triples}

Recall that the renormalization operator $\cren \colon \mathcal U\to \mathcal D_{\eps}^{cr, \bbR}$ leaves invariant the horseshoe-like set $\mathcal I$. Consider a small open neighborhood $\mathcal V$ of $\mathcal I$  in the space $\mathcal D_\eps$ so that the function $q_N(f)$ extends to  $\mathcal V$ as a locally constant function. Now $q_N$ is defined even for non-$\bbR$-preserving $f\in\mathcal V$.

If $\mathcal V$ is sufficiently small, then for each map $f\in \mathcal V$,  there exists a fundamental crescent $C=C_f$ and the perturbed Fatou coordinate $\Psi=\Psi_f$ with $\Psi(0)=0$ that conjugates $f^{q_N}$ to the unit translation  in $f^{q_N}(C)\cup C$, and is real-symmetric for real-symmetric $f$.  Thus the operator $\cren$ extends to $\mathcal V$. Moreover, due to the compactness of  $\bar  {\mathcal I}$, all maps $\cren f$ for $f\in \mathcal V$ extend analytically to the same strip $\Pi_c$ around the real axis, that is, $\cren \colon \mathcal V\to \mathcal D_{c}$ is well-defined.

The construction of renormalization $\cren \colon \mathcal U\to \mathcal D_{\eps}^{cr}$ for critical circle maps motivates the following definition of the renormalization in the space of triples, also used in \cite{GorYa}.
\begin{definition}
\label{def-R}
The renormalization in the space of triples $D$ is defined in the following way: for any triple $(F, H_a, G)\in p^{-1}(\mathcal V)$, the triple $\mathcal R (F, H_a, G)=(\tilde F, H_{\tilde a}, \tilde G)$ is given by 
 $$\tilde G = \Psi'(0)  G \Psi^{-1} \text{ restricted to } \bar \Pi_\eps,$$
  $\tilde a = a (\Psi'(0))^2$ so that $$H_{\tilde a} = \Psi'(0)^3 \cdot H_a (\Psi'(0))^{-1},$$
  and
 $$\tilde F = \cren f  \circ (\pi_{\tilde a, \tilde G} H_{\tilde a} \tilde G)^{-1}  \text{ restricted to } \Pi_{\eps'}, $$
where $f = p(F, H_a, G)$.
 \end{definition}
Since $ \cren f =\Psi \circ P \circ \Psi^{-1}$ where $P$ is an iterate of $f = F \pi_{a,G}H_a  G $, the composition that gives $ \cren f  $ will end with terms  $H_a G \Psi^{-1} = \Psi'(0)^{-3} H_{\tilde a} \tilde G$. Since $\pi_{a,G}$ are bijections, the inverse map  in the definition of $\tilde F$ does not produce any branching and $\tilde F$ is well-defined.

 Note that $\tilde F$ commutes with the unit translation as the compositional difference of two maps $\cren f$ and  $\pi_{\tilde a, \tilde G} H_{\tilde a} \tilde G_a$ which both commute with the unit translation.

Recall that the construction of the map $(a,G)\to \pi_{a,G}$, and thus the construction of $p$, depends on the set $\Omega^{\bbR}$ of maps $G$ with certain uniform estimates. The resulting operator $\mathcal R$ will only be defined on a certain subset of the space $D$ of triples where the estimates on the $G$-components hold true. This set  $\Omega^{\bbR}$ will be fixed in Lemma \ref{lem-horseshoe}.

\begin{lemma}
\label{lem-projections}
 For any triple $\hat f$ such that $p(\hat f)\in \mathcal V$, we have $p (\mathcal R \hat f) = \cren(p(\hat f))$ whenever the left-hand side is defined, i.e. the renormalization operator $\mathcal R$  projects to the renormalization operator $\cren$ on $\mathcal V$.
\end{lemma}
\begin{proof}
Indeed,
$$ p  ( \mathcal R (F, H_a, G)) = \tilde F \circ \pi_{\tilde a, \tilde G} H_{\tilde a} \tilde G = \cren f =\cren (p (F, H_a, G)).$$

\end{proof}

Recall that the construction of the renormalization operator $\cren$ depends on the choice of the iterate $q_N$.
In the assumptions of Lemma \ref{lem-pi-exist}, let   $U_{\delta, \Omega}$ be the set of triples $(F,H_a, G)$ such that $0\le |a|\le \delta,\; G\in \Omega$.

 \begin{lemma}
\label{lem-compact}
For sufficiently small  $\eps>0$ and a suitable choice of $\eps'>0$ in the definition of the space of triples $D$, for a set  \begin{multline}
\Omega^{\bbR} = \{G\colon \bar \Pi_{\eps}\to \bbC, G(0)=0, G'(0)=1, G\text{ preserves } \bbR  \mid \\  \sup_{\bar  \Pi_\eps}|G''|\le C, \min_{[-0.5, 0.5]}G'\ge 1/C\}\end{multline} with any sufficiently large $C$,
 we can choose its neighborhood $\Omega\supset \Omega^{\bbR}$, $N$,  and a small $\delta$ such that the renormalization operator $\mathcal R\colon D\to D$  is defined, analytic, and has a compact differential on  the ``truncated fibers'' $U_{ \delta,  \Omega} \cap p^{-1}(\mathcal V)$ of the projection $p$.

The images of truncated fibers  $\mathcal R(U_{\Omega,\delta} \cap p^{-1}(\mathcal V))$ are precompact, and the projection $p \colon \mathcal R(U_{\Omega,\delta} \cap p^{-1}(\mathcal V)) \to \mathcal D_{\eps}$ is defined.

Finally, the set $\mathcal R(U_{\Omega^{\bbR},0} \cap p^{-1}(\mathcal V))$ is contained in $U_{\Omega^{\bbR},0}$.

\end{lemma}
\begin{proof}

Note that for any $f\in \mathcal V$, the renormalized map  $\cren f$ is defined in a certain strip  $\Pi_{c}$; choose $\eps<c/2$, then the map $\cren f$ is defined in $\Pi_{2\eps}$. We will use this $\eps$  in the definition  of $D$. The value of $\eps'$ will be chosen later.

Since the map $\pi_{a,G}$ is analytic on $a, G$, the projection $p$ is analytic. Since the uniformizing coordinate  $\Psi=\Psi_f$  is defined in $\mathcal V$ and  analytic on $f$, the operator $\mathcal R$ is defined and analytic on the intersection of the domain of $\pi_{a,G}$ with $p^{-1}(\mathcal V)$.

Now let us check that its differential is compact.
Lemma \ref{lem-Psi-estim} implies that  for sufficiently large $N$, the map   $\tilde G= \Psi'(0) G \Psi^{-1}$ is defined in a rectangle  $\{-1<\Re z <1, -2\eps<\Im z<2\eps\}$ that covers $\bar  \Pi_\eps$. It remains to show that for a suitable choice of $N$,  $\eps'$, and $\Omega$, the map $\tilde F$ also extends analytically to a domain that contains $\Pi_{\eps'}$ in its interior. This will imply both compactness statements of the lemma.

Let $(\tilde F, H_{\tilde a}, \tilde G) = \mathcal R(F, H_a, G)$.
Note that   $$\tilde G'(z) = \Psi'(0) G'|_{\Psi^{-1}(z)} \frac{1}{\Psi'(\Psi^{-1}(z))}$$ and
\begin{equation}
 \label{eq-G}
 \tilde G''(z) = \Psi'(0) G''|_{\Psi^{-1}(z)} \frac{1}{\Psi'(\Psi^{-1}(z))^2} - \Psi'(0) G'|_{\Psi^{-1}(z)} \frac{\Psi''|_{\Psi^{-1}(z)} }{\Psi'(\Psi^{-1}(z))^3}.
\end{equation}

Let $$\dis \Psi = \sup_{\mathcal V}\sup_{ a, b\in \Psi^{-1}(\bar \Pi_{\eps})}\left(\frac{\Psi'(a)}{\Psi'(b)},  \frac{\Psi''(a) }{\Psi'(a)^2}\right)$$ be the maximal distortion of $\Psi$; due to Lemma \ref{lem-Psi-estim}, this is finite and independent of $N$. Consider any $C$ with $C>2 (\dis \Psi)^2$, and use it in the definition of  $\Omega^{\bbR}$.

Choose $N$ so that $\Psi'(0)> 2 C \cdot \dis \Psi$; this is possible due to Lemma \ref{lem-Psi-estim}.  
Due to the normalization $G(0)=0$ and $G'(0)=1$, we conclude that if $G\in \Omega^{\bbR}$, then in $\bar  \Pi_{\eps}$,   $$|\tilde G'|> (\dis \Psi)^{-1} \cdot |G'|_{\Psi^{-1}(z)}| > (\dis \Psi)^{-1} (1-\frac{C \cdot \dis \Psi}{|\Psi'(0)|}) > 0.5 (\dis \Psi)^{-1} > \frac{1}{C}$$ and $$|\tilde G''|<C\frac{(\dis \Psi)^2}{ |\Psi'(0)|}+ (\dis \Psi)^2 (1 + \frac{C \cdot \dis \Psi}{|\Psi'(0)|}) < 0.5 (\dis \Psi)^2 + 1.5 (\dis \Psi)^2<C.$$

Thus $G\in \Omega^{\bbR}$ implies $\tilde G\in \Omega^{\bbR}$, and we have proved the last statement of the lemma. 

Use Lemma \ref{lem-pi-exist} to define $\pi_{a,G}$, choose $\delta,  \Omega\supset \Omega^{\bbR}$, and  $\eps'=\tau_2/2$ such that for all $G\in \Omega$, we have $\pi_{0, G}H_0 G(\bar \Pi_{\eps})\supset \Pi_{2\eps'} $. We will use this $\eps'$ in the definition of the space of triples  $D=D_{\eps, \eps'}$.

Since  $\cren f = \tilde F\circ \pi_{\tilde a,\tilde G}\circ H_{\tilde a}\circ \tilde G$ is defined in $\Pi_{2\eps}$, we conclude that $\tilde F$ analytically extends to $\pi_{\tilde a,\tilde G}\circ H_{\tilde a}\circ \tilde G(\bar  \Pi_{\eps}) \supset \Pi_{2\eps'}$.



Since we restrict $\tilde F, \tilde G$  to smaller domains $\bar \Pi_\eps, \Pi_{\eps'}$ in the definition of $\mathcal R$, this operator $\mathcal R$ has compact derivative. We have also proved that the set $\mathcal R(U_{\Omega, \delta}\cap p^{-1}(\mathcal V))$ is pre-compact.

The above arguments  imply that for sufficiently small $\Omega\supset \Omega^{\bbR}$ and $\delta$, on the image of $\hat f\in U_{\Omega, \delta}\cap p^{-1}(\mathcal V)$ under $\mathcal R$, the map $\pi_{\tilde a, \tilde G} $ is defined and $p(\mathcal R \hat f)=\pi_{\tilde a,\tilde G}\circ H_{\tilde a}\circ \tilde G$ is well-defined in a certain strip $\Pi_{\tau_1}$ around $\bbR/\bbZ$. This map extends to the strip $\Pi_\eps$ since it equals $\cren f$. Thus $p\colon \mathcal R(U_{\Omega, \delta}\cap p^{-1}(\mathcal V)) \to \mathcal D_{\eps}$ is defined. This completes the proof. 
\end{proof}

\begin{lemma}
\label{lem-fibers}
In the assumptions of Lemma \ref{lem-compact}, the operator $\mathcal R$ in the space of triples uniformly contracts on truncated fibers $p^{-1}(f) \cap U_{\Omega, 0}$ of the projection $p$ for critical maps $f\in \mathcal V\cap \mathcal D^{cr}_\eps$.

In more details, for any set $\Omega^{\bbR}$ as in Lemma \ref{lem-compact} we can choose its neighborhood $\Omega$, $N$, and constants $0<\lambda<1$, $c$, such that for any two triples $(F_1, H_0, G_1)$ and $(F_2, H_0, G_2)$ in $U_{\Omega, 0}$  with $$p(F_1, H_0, G_1) = p(F_2, H_0, G_2)\in \mathcal V\cap \mathcal D^{cr}_{\eps},$$  we have
$$\dist (\mathcal R^n (F_1, H_0, G_1), \mathcal R^n (F_2, H_0, G_2)) <c \lambda^n \dist((F_1, H_0, G_1), (F_2, H_0, G_2)). $$
\end{lemma}
\begin{proof}


The estimate \eqref{eq-G} applied to $G_1, G_2$ and Lemma \ref{lem-Psi-estim} implies  that by choosing sufficiently large $N$ in the definition of $\cren$, we can guarantee that \begin{equation}\label{eq-G''}\dist (\tilde G_{1}'', \tilde G_{2} '') \le \mu_0\dist_{C^2} (G_1, G_2)\end{equation} for some $\mu_0<1$. Indeed, the first summands in \eqref{eq-G} are uniformly small, and in the second summands, $G_1'-G_2'$ is computed on a uniformly small neighborhood of zero and multiplied by a bounded function. Since $G_1'(0)=G_1'(0)=1$, we can estimate $G_1'-G_2'$ on a small neighborhood of zero in terms of $G_1''-G_2''$ and get \eqref{eq-G''}.

Since  we only consider triples normalized by  $G(0)=0$ and $G'(0)=1$, this implies that $G$-components of the triples with the same projection approach each other exponentially quickly.

Since $G$-components of $\mathcal R^n(F_{1,2}, H_0, G_{1,2})$ approach each other exponentially quickly, the same holds for the corresponding maps $\pi_{0,G_{1,2}}$ due to the fact that the dependence of $\pi_{0,G}$ on $G$ is analytic, thus Lipshitz on $\Omega^{\bbR}$.  Since the triples have the same projection, their $F$-components approach each other exponentially quickly. This implies the statement.

%
%
%
\end{proof}

Now we can lift $\mathcal I$ to the horseshoe-like set $\Lambda$ in the space of triples. In this lemma, we will fix the choice of $\eps, \eps'$ in the definition of $D$. We will also fix the choice of $\pi_{a,G}$, and thus complete the construction of the projection $p$.

\begin{lemma}
\label{lem-horseshoe}
There exists a choice of $\eps, \eps'$ in the definition of $D$, a choice of the map $(a,G)\to \pi_{a,G}$, and a set $\Lambda \subset D_0$ such that $\Lambda$ is invariant under $\mathcal R$, the operator $\mathcal R$ is an analytic operator with compact derivative on a neighborhood of $\Lambda$, and the projection $p\colon \Lambda \to {\mathcal I}$ is a homeomorphism that conjugates the dynamics of $\cren|_{\mathcal I}$ to the dynamics of $\mathcal R|_{\Lambda}$.

All triples in $\Lambda$ are $\bbR$-preserving.
\end{lemma}
\begin{proof}
For each critical circle map $f\in \mathcal I$, construct its lift $j(f) = (F,H_a, G)$ to $D$ using Lemma \ref{lem-lift-cr}: namely, set $$G(z) = \sqrt[3]{f(z)-f(0)}/C\text{, and }F(z) = z+f(0).$$ Then   $p((F, H_0,G))=f$.

All resulting maps $G$ for all $f\in \mathcal I$ are $\bbR$-preserving. Choose $\eps$ small, so that  all these maps   extend to $\bar \Pi_\eps$ univalently with uniform estimates on $\max_{\bar \Pi_\eps} |G''|$, $\min G'|_{[-0.5, 0.5]}$; consider  $C$ so that the set $\Omega^{\bbR}$ from Lemma \ref{lem-compact} contains all $G$-components of all lifts of critical maps  $f\in \mathcal I$. Decrease $\eps$ if needed to satisfy assumptions of Lemma \ref{lem-compact}, and use this lemma to choose $\eps'$. These values of $\eps, \eps'$ will be used in the definition of the space $D$. Finally, fix $\delta$ and $N$ so that assertions of Lemma \ref{lem-compact} and Lemma \ref{lem-fibers} hold true.

The space of triples $D$ is now fixed, and the projection $p$ is defined.

For each $f\in \mathcal I$, consider critical circle maps $f_{-n}\in \mathcal I$ such that $$\cren^n f_{-n}=f;$$ note that this uniquely defines $f_{-n}$.

For each $f_{-n}$, consider  truncated fibers of the projection, $p^{-1}(f_{-n})\cap U_{\Omega,0}$. Due to the choice of $\Omega^{\bbR}$, these sets are non-empty for all $g\in \mathcal I$: every such set contains a lift $j( f_{-n})$ of $f_{-n}$.

Since $\mathcal R$ contracts on truncated fibers (Lemma \ref{lem-fibers}), the diameters of the sets $\mathcal R^n(p^{-1} f_{-n}\cap U_{\Omega, 0}) \subset p^{-1}(f)$ tend to zero. All these sets are non-empty since they contain $\mathcal R^n (j( f_{-n}))$. Since each set is precompact and components of the triples in these sets extend to larger strips (Lemma \ref{lem-compact}), there exists a unique accumulation point $i( f)$ of these sets, and $F,G$- components of the triple $i(f)$ are defined in strips larger than $\Pi_\eps$, $\Pi_{\eps'}$ respectively. By continuity, the projection $p$ is defined on $i(f)$ and $p(i(f))=f$.

The construction implies that $$i(\cren f) = \mathcal R i( f),$$ thus the set $\Lambda = i(\mathcal I)$  is an invariant set for $\mathcal R$, and the dynamics of $\mathcal R|_{\Lambda}$ is conjugate to the dynamics of $\cren|_{\mathcal I}$ via the continuous map $p$.

Since  $i(f)$ is the limit point of $ \bbR$-preserving triples $\mathcal R^n(j( f_{-n}))$,  it is $\bbR$-preserving.

All maps $\mathcal R^n(p^{-1}(f_{-n})\cap U_{\Omega^{\bbR},0})$ belong to $U_{\Omega^{\bbR},0}$ due to the last statement of Lemma \ref{lem-compact}. In particular, $\Lambda\subset U_{\Omega^\bbR, 0}$.   Thus Lemma \ref{lem-compact} implies that $\mathcal R$ is an analytic operator with a compact differential on a neighborhood of $\Lambda$.

It remains to prove that $i$ is continuous, that is, the projection $p$ is a homeomorphism between $ \Lambda$ and $\mathcal I$.

Let $f\in \mathcal I$, $\delta>0$; show that the lifts of sufficiently close points of $\mathcal I$ are $\delta$-close to $i(f)$.  Let $\lambda, c$ be the same as in Lemma \ref{lem-fibers}; choose $n$ so that $c\lambda^n\cdot 2C<\delta/3$. For a map $f_{-n}$, take its neighborhood $U$ in $\mathcal I$, and let $g\in U$. Construct the lifts $j( g), j( f_{-n})$ using Lemma \ref{lem-lift}; clearly, the lifts are close in the space of triples if $U$ is small enough. Due to Lipshitz estimates on $\mathcal R$, we also have that $\mathcal R^n (j( g))$ and $\mathcal R^n(j( f_{-n}))$ are $\delta/3$-close if $U$ is small enough.

Note that the diameters of the sets $$\mathcal R^{n} (U_{\Omega,0}\cap p^{-1}(g))\text{ and }\mathcal R^{n} (U_{\Omega,0}\cap p^{-1}(f_{-n}))$$ are smaller than $c\lambda^n \cdot \diam (U_{\Omega,0})$ due to Lemma \ref{lem-fibers}, and the latter is smaller than $\delta/3$ due to the choice of $n$. Also, the closure of the first set contains both $\mathcal R^n(j( g))$ and $i(\cren^ng)$, and the closure of the second set contains both $\mathcal R^n(j( f_{-n}))$ and $i(f)$.
Since $\mathcal R^n (j( g))$ and $\mathcal R^n(j( f_{-n}))$ are $\delta/3$-close, we conclude that $i(f)$ and $i(\cren^ng)$ are $\delta$-close for $g\in U$. This completes the proof.


\end{proof}
The previous lemmas provide the description of the Banach space $D$, its projection $p$, and the operator $\mathcal R$, required for Theorem \ref{th-operator}, with the exception of the last property: uniform hyperbolicity of $\mathcal R$ with two unstable directions.
The next lemma completes the proof of Theorem \ref{th-operator} and provides an explicit formula for the maximal expansion rate along the second one-dimensional unstable distribution of $\mathcal R$.

\begin{lemma}
\label{lem-distributions}
 Along each orbit in $\Lambda$, the operator $\mathcal R$ is hyperbolic, with two-dimensional unstable distribution which intersects $TD_0$ on a one-dimensional distribution $\eta_{\hat f}$. The operator $\mathcal R$ is uniformly hyperbolic on $\Lambda$.

 The stable distribution of $\mathcal R$ is a full preimage under $(dp)$ of the stable distribution of $\cren$. 

The distribution  $(dp)\eta_{\hat f}$ is  the unstable distribution of $\cren$.
At the orbit of the triple $\hat f\in \Lambda$,  the minimal expansion rate $\lambda_1$ of $\mathcal R$ along $\eta_{\hat f}$ coincides with the minimal expansion rate of $\cren$ along its unstable distribution $(dp)\eta_{\hat f}$ at the orbit of $p(\hat f)$.

Set
\begin{equation}\label{eq-lambda2}\lambda_2 =\limsup_{n\to +\infty} \sup_{m\ge 0} \sqrt[n]{\Psi'_{m, n}(0))^2}
\end{equation}
 where $\Psi_{m, n}$ is the uniformizing chart that corresponds to the iterate $\mathcal R^n $ at $\mathcal R^m \hat f$; if $\lambda_1>\lambda_2$, then $\mathcal R$ has an invariant one-dimensional distribution $\xi_{\hat f}$ transversal to $D_0$ along the orbit of $\hat f$, and the maximal expansion rate along this distribution coincides with $\lambda_2$.

Both one-dimensional distributions $\eta_{\hat f}$ and $\xi_{\hat f}$ are generated by  vector fields that belong to $TD^{\bbR}$. 

\end{lemma}
\begin{proof}
Roughly speaking, since $\mathcal R$ contracts on the fibers of the projection $p$ and $\mathcal R|_{D_0}$ projects, via $p$, to the hyperbolic operator $\cren$ on $\mathcal D^{cr}_{\eps}$, the operator $\mathcal R$ must be hyperbolic on $D_0$ with one unstable direction that projects to the unstable direction of $\cren$ and has the same expansion rate. Since $\mathcal R^n$ takes $a$ to $(\Psi'_{m,n}(0))^2 a$, it expands in the direction transversal to $D_0=\{a=0\}$, thus it must have one more unstable direction in $D$ with maximal expansion rate $\lambda_2$ as above.

Below we provide detailed proofs of these statements, by constructing the stable foliation and unstable invariant cone fields for $\mathcal R$.

 \textbf{Stable distribution of $\mathcal R$}

 Let $W^s_f\subset \mathcal D_{\eps}^{cr}$ be the local stable manifold of $f\in \mathcal I$ under $\cren$. We will prove that an analytic manifold $p^{-1}(W^s_f)$ is a stable manifold of $i(f)$. Consider the lift $j ({W^s_f})$ constructed as in Lemma \ref{lem-lift-cr}.  Then $p^{-1} (W^s_f)$ is a saturation of $j ({W^s_f})$ by the fibers of $p$, and one of these fibers  passes through both $j( f)$ and $i(f)\in \Lambda$. Since $\mathcal R$ uniformly contracts on the truncated fibers and both $i(f)$, $j(f)$ belong to $U_{\Omega, 0}$, it is sufficient to prove that $p^{-1} (W^s_f)$ is a locally stable manifold of its point $j(f)$.
 
 Let $v$ be any vector of unit length such that $(dp)v\in TW^s_f$. We prove that $v$ is uniformly contracted under the action of $d\mathcal R$. Replacing $\mathcal R$ with its iterate, we may and will assume that $\|d\cren^n|_{W^s_f}\| <\tau^n $ and $\|d\mathcal R^n|_{p^{-1}(f)}\|<\mu^n$ for all $n=1, 2, \dots$ and some $\mu, \tau\in(0,1)$. Let us show that the contraction rate of $d\mathcal R$ on $v$ is at least $\max(\mu, \tau)$.

There exists a uniform $k$ such that for each $v\in T_{\hat f} D$, we can construct a vector $u$ with the same projection, $(dp)u=(dp) v$, such that $\|u\|<k \|(dp)v\|$; indeed, we can take $u = (dj)(dp)v$.

 Take  $c\gg 2k\|dR\|$ and consider the cones $$\mathcal C^s_{\hat f} = \{\|v \| > c \|(dp)v \| \mid v \in T_{\hat f}D\}$$ ``along'' the fibers of the projection $p$ for all $\hat f\in \Lambda$.   Then  for any vector field in this cone, we have a representation $v=u +w$ where $w\in \Ker (dp)$ and $u$ is as above, $\|u\|\le k\|(dp)v\|\le \frac kc \|v\|$. Thus $$\|w\|\le \|v\|+\|u\| \le \|v\|(1+k/c),$$ and  $$  \|dR v\| \le  \|d\mathcal R   w\| + \|d\mathcal R u\| \le  \mu \|v\|(1+k/c) + \|d\mathcal R\| \cdot  \|v\|\cdot  \frac kc. $$ Hence for any $\eps$, for sufficiently large $c$, we have $\|d\mathcal R v\|< (\mu+\eps)\|v\| $: the contraction rate in the cones  $\mathcal C_{\hat f}^s$ is only slightly bigger than $\mu$.

Now, show that  $\|d\mathcal R^n v\|< C (\max (\tau, \mu) +\eps)^n \|v\|$. If all vectors $d\mathcal R^n v$ for all $n$ stay in the corresponding cones $\mathcal C_{\mathcal R^n \hat f}^s$, the statement follows from the above estimate. Otherwise, consider a block $n, n+1, \dots, n+m$ such that $d\mathcal R^n v$ does not belong to the cone $\mathcal C^s_{\mathcal R^n \hat f}$, but the vectors  $d\mathcal R^{n+1} v, \dots, d\mathcal R^{n+m} v$ do. It is sufficient to prove an estimate for each (finite or infinite) block of this form. 

Since $d\mathcal R^n v $ is not in the cone, we have $\|d\mathcal R^n v\|<c   \|(dp)d\mathcal R^n v \|< \tau^n \|dp(v)\| < C \tau^n \cdot \|v\|$ and the statement holds for $n$. Further, $$\|d\mathcal R^{n+j} v\| < (\mu+\eps)^{j-1} \cdot \|d\mathcal R\| \cdot \|d\mathcal R^n v\| < \tilde C\tau^n (\mu+\eps)^{j-1}$$ for $j=1, \dots, m$, thus the estimate holds for the block $n, n+1, \dots, n+m$.

This produces the codimension-1 stable distribution for $\mathcal R$ in $D_0$.

\textbf{One-dimensional unstable distribution for $\mathcal R|_{D_0}$.}

 We will construct unstable cones for this operator. Let $f\in \mathcal I$, $\hat f = i(f)\in \Lambda$. Let $TW^u_f$ be  the unstable distribution for $\cren$. Let $p_{u}\colon T\mathcal D_{\eps}^{cr}\to T\mathcal D_{\eps}^{cr} $ be the projection onto $TW^u_f$  along the stable distribution $TW^s_f$ , and define uniformly bounded linear functionals $L_{\hat f}$ by  $L_{\hat f}= p_u \circ dp$. Let $u_{\hat f} \in  (dj)(TW^u_f)$ be the vectors that belong to the lift of the distributions  $TW^u_f$  to the space $D_0$ constructed as in Lemma \ref{lem-lift-cr}, and normalize them so that $L_{\hat f} u_{\hat f}=1$; note that the norms $\|u_{\hat f}\|$ are uniformly bounded.

 Then $\Ker L_{\hat f}$ is a stable distribution for $\mathcal R$, thus replacing $\mathcal R$ with its iterate we can achieve $$\|d\mathcal R|_{\Ker\, L_{\hat f}}\|<\lambda_3<1;$$ $TW^u_f$  is the unstable distribution for $\cren$, thus  replacing $\mathcal R$ with its iterate we can guarantee that $$L_{\mathcal R {\hat f} }(d\mathcal R u_{\hat f}) =p_u (d\cren (dp) u_{\hat f})  >\lambda_1\gg 1.$$

 Now, it is easy to check that the cones  $$\mathcal C^{uu}_{\hat f} = \{|L_{\hat f}v|>c\|v\|\}$$ are invariant unstable cones for $\mathcal R$ for sufficienly small uniformly chosen $c$. Note that these cones project to the unstable cones of $\cren$.

 Indeed, let $v$ belong to this cone, and represent it as $v = L_{\hat f}(v) \cdot u_{\hat f}+w $ where $w$ belongs to $\Ker L_{\hat f}$. Then $$\|w\|\le \|v\| + |L_{\hat f}(v) |\cdot \|u_{\hat f}\| \le  |L_{\hat f}(v) | (\frac 1c + \|u_{\hat f}\|), $$    $$|L_{\mathcal R{\hat f}} (d\mathcal R v)| = |L_{\hat f}(v) \cdot L_{\mathcal R{\hat f}}d\mathcal R u_{\hat f}|\ge   |L_{\hat f} (v) | \cdot \lambda_1$$ and $$c\|d\mathcal Rv\| = c\| L_{\hat f}(v) \cdot d \mathcal R u_{\hat f} +d\mathcal R w \|  \le  c|L_{\hat f}(v)| \cdot (\|d\mathcal R\|\cdot \|u_{\hat f}\| +\lambda_3 \left(\frac 1c + \|u_{\hat f}\|)\right) ,$$ hence the cones $\mathcal C^{uu}_{\hat f}$ are invariant for small uniform $c$. Also, since   $\|d\mathcal R^nv\|  \ge  |L_{\hat f}(v)| \cdot \| d \mathcal R^n u_{\hat f} \| - \lambda_3^n (\frac 1c + \|u_{\hat f}\|))$,  by using this inequality for sufficiently large iterate of $\mathcal R$ we conclude that high iterates of $\mathcal R$ uniformly expand in the cones.

Now, the unstable distribution $\eta_{\hat f}$ is constructed as the intersection of images of the unstable cones $\mathcal C^{uu}_{\hat f}$.
 Since the projections of the cones (and their iterates) to $T\mathcal D^{cr}_{\eps}$ are unstable cones (and their iterates) for $\cren$,  the unstable distribution for $\mathcal R$ projects to the unstable distribution for $\cren.$

 Since in the unstable cones, we have both lower and upper estimates on $\|v\|$ in terms of  $L_{\hat f} (v)$, and thus in terms of $\|(dp)v\|$, the minimal expansion rates along these distributions coincide.

\textbf{Two-dimensional unstable distribution for $\mathcal R$.}

Let $\tilde \eta_{\hat f}$ be unit vectors that belong to the distribution $\eta_{\hat f}$. 
Let $v\in T_{\hat f} D$, ${\hat f}\in D_0$,  be represented as a sum $v=x \cdot \tilde \eta_{\hat f} + y \cdot \frac{\partial }{\partial a}+z $, where $z\in T W^s_{\hat f}$ and $\frac{\partial }{\partial a}$ is a change of the coordinate $a$ in the space of triples. Recall that $a$ is multiplied by $\Psi'(0)^2$ under $\mathcal R$. Due to Lemma \ref{lem-Psi-estim}, by replacing $\mathcal R$ with its iterate we can guarantee that $\Psi'(0)^2>\mu >1$. Thus $$d \mathcal R (\frac{\partial }{\partial a}) = k \tilde \eta_{\hat f} + l \frac{\partial }{\partial a} +m $$ where $|l|>\mu$. Since $\|dR\|$ is uniformly bounded, there is a uniform bound $\max(|k|, |l|, \|m\|) <K$. Also, $\|d\mathcal R z\| < t \|z\|$ for $t<1$, due to the uniform contraction on $ W^s_{\hat f}$. Now, it is easy to check that
$$\mathcal C^u_{\hat f} = \{v=x \cdot \tilde \eta_{\hat f} + y \cdot\frac{\partial }{\partial a}+z \mid   |x|+c|y|>\|z\|\}$$
is a family of  unstable invariant cones for $d\mathcal R$ if $ c$ is chosen sufficiently big.  This  implies that $d\mathcal R$  has  a two-dimensional unstable distribution transversal to $D_0$ that  contains $\eta_{\hat f}$.

In more detail, assume $d\mathcal R \tilde \eta_{\hat f} = \lambda_{\hat f}  \tilde \eta_{\mathcal R{\hat f}}$ where $|\lambda_{\hat f}|>1$.
  The image of the vector $v=x \tilde \eta_{\hat f}+y \frac{\partial }{\partial a}  +z$ has coordinates  $$\tilde x = \lambda_{\hat f} x+yk,\; \tilde y = ly, \; \tilde z = d\mathcal R z +  y m .$$

  Now, $$\|d\mathcal R z + y m\| <t \|z\|+K|y| < t (|x|+c|y|)+K|y| $$ and we also have
  $$|\lambda_{\hat f} x + y k|+c|l y| > |\lambda_{\hat f}|\cdot | x|+ |y| (c\mu-K) .$$ For a large $c$, the second quantity is bigger and thus the cone field is invariant.

  It remains to prove that the vectors in the cone are uniformly expanded. Indeed, for large $c$ and small $\delta>0$,
\begin{eqnarray*}
  |\lambda_{\hat f} x+yk| + c| l y|+ \frac \delta 2\|d\mathcal R z + y m\|  &>&  (1+\delta) (|x|+c|y|)\\&>& \left(1+\frac \delta 3\right)\left(|x|+c|y| +\frac \delta 2\|z\|\right) ,
\end{eqnarray*}
  that is, the vectors are uniformly  stretched in the norm $|x|+c|y| + \frac{\delta}{2} \|z\|$.
  This implies the existence of the two-dimensional unstable distribution $l_{\hat f}$ that uniformly expands under the action of $\mathcal R$ and depends continuously on a point $\hat f\in \Lambda$. Thus $\mathcal R$ is  uniformly hyperbolic on $\Lambda$.

 \textbf{Second one-dimensional unstable distribution for $\mathcal R$.}

  If $\lambda_1>\lambda_2$, the existence of the unstable distribution uniformly transversal to $D_0$ follows from standard results on operators in  $\bbR^2$. In more detail, one can choose bases in the two-dimensional unstable distributions $l_{\hat f}, l_{\mathcal R\hat f}, l_{\mathcal R^2 \hat f}, \dots $ for $\mathcal R$ so that the first basis vector belongs to the distribution  $\eta_{\mathcal R^n \hat f}$ and its length is uniformly bounded above and below,  while  the second basis vector has a unit projection to $\frac{\partial}{\partial a}$ and its length is uniformly bounded above and below. Then the matrices of $d\mathcal R^n|_{l_{\mathcal R^m \hat f}}$ in these bases have the form $\begin{pmatrix}\lambda_{m,n} & \tau_{m, n} \\ 0 & \mu_{m,n}\end{pmatrix}$ with 
  $$\frac 1C \|d\mathcal R^n|_{\eta_{\mathcal R^m \hat f }}\|<\lambda_{m,n} < C \|d\mathcal R^n|_{\eta_{\mathcal R^m \hat f }}\|\text{ and  }\mu_{m, n} = (\Psi_{m, n}'(0))^2.$$

Due to the definition of $\lambda_1, \lambda_2$, there exists $n$ such that for all $m$, $\lambda_{m, n}>  (\lambda_1-\eps)^n >(\lambda_2+\eps)^n > \mu_{m, n}$. Fix $n$ and increase the lengths of  second basis vectors in $l_{\mathcal R^k \hat f}$ so that $$\tau_{m, n} < (\lambda_1-\eps)^n - (\lambda_2+\eps)^n < \lambda_{m, n}-\mu_{m,n};$$ this is possible since $\|d\mathcal R^n\|$, and thus $\tau_{m,n}$, are uniformly bounded.

Now it is easy to see that  the cones $|x|<|y|$ in $l_{\mathcal R^{m}\hat f}$ are invariant  under $(d\mathcal R)^{-n }$, thus under $d\mathcal R^{-nk}$ for all $k$ (note that $\mathcal R$ is invertible on $l_{\mathcal R^m \hat f}$). Consider their images under $d\mathcal R^{-kn}$ in $l_{\hat f}$. The intersection of this sequence of embedded cones contains vectors $v$ such that their images under $d\mathcal R^{kn}$ stay in $|x|<|y|$, and thus $\|d\mathcal R^{kn} v\|<c \mu_{0,kn}\|v\|$; since $\lambda_{m,n} \gg \mu_{m,n}$, this vector $v\in l_{\hat f}$ is unique up to rescaling. Its images  under $d\mathcal R$ provide us with an invariant one-dimensional distribution in $l_{\mathcal R^k \hat f}$ uniformly transversal to $\eta_{\hat f}$, such that the maximal expansion rate along this distribution equals $$\underset{n\to +\infty}{\limsup} \sup_{m\ge 0} \sqrt[n]{\mu_{m, n}} =  \lambda_2.$$

\textbf{One-dimensional unstable distributions are generated by real vector fields}

Note that the above constructions of unstable cone fields  work equally well in the space of $\bbR$-preserving triples $D^{\bbR}$. In particular, $\eta_f$ is generated by an  $\bbR$-preserving vector field, and $j$ lifts $\bbR$-preserving maps to $\bbR$-preserving triples. Thus the resulting unstable direction fields are generated by some vectors from $TD^{\bbR}$.

 \end{proof}

The following lemma removes the assumption $\lambda_1>\lambda_2$ using the properties of $(-1)$-measure; it will not be used in the proofs of our main results, since  the assumption $\lambda_1>\lambda_2$ is still necessary to prove smoothness of Arnold tongues. 
\begin{lemma}
\label{lem-weakstable}
The statements of the previous lemma hold true without the assumption $\lambda_1>\lambda_2$: the one-dimensional unstable distribution $\xi_{\hat f}$ exists for all $\hat f\in \Lambda$ and belongs to the codimension-1 subspace $L_{\hat f} v=0$ given by
$$L_{\hat f} v =  \int_{\bbR/\bbZ} (dp)v|_{(p(\hat f))^{-1}(z)} d\mu_{p(\hat f)}=0.$$ 
\end{lemma}
\begin{proof}
Note that  codimension-1 spaces $\{L_{\hat f} v=0\}$ are preimages under $p$ of the spaces $\{L_f v=0\}$ in $\mathcal D_{\eps}$. Recall that for circle diffeomorphisms, the condition $L_f v=0$ defines tangent spaces to $\{\rot f=\alpha\}$, thus $\mathcal R$ takes $\{L_{\hat f} v=0\}$ to $\{L_{\mathcal R\hat f} v=0\}$ if $p(\hat f)$ is a circle diffeomorphism. The limit transition, possible due to Theorem \ref{th-C1}, implies that the same holds if $p(\hat f)$ is a cubic critical circle map, in particular for all  $\hat f \in \Lambda$. Also, Theorem \ref{th-C1} implies that $\{L_{\hat f}v=0\}$ intersects $D_0$ on the stable subspace of $\mathcal R$, thus is uniformly transversal to $\eta_{\hat f}$.

The previous lemma implies that $\mathcal R$ has a two-dimensional unstable  distribution $l_{\hat f}$. The intersection of $l_{\hat f}$ with $\{L_{\hat f} v=0\}$ is a one-dimensional invariant distribution $\xi_{\hat f}$ that is uniformly transversal to $\eta_{\hat f}$. Since $\mathcal R$ multiplies the projection of any vector to $d/da$ by $\Psi'(0)$, the maximal expansion rate along $\xi_{\hat f}$ is the same as in the previous lemma. This completes the proof.
\end{proof}

\section{Uniform estimates on stable manifolds far from critical circle maps}
\label{sec-derivatives}

Let $\hat {\mathcal  V}_\alpha\subset  D$ be a germ of the manifold $\{\rot p(\hat f) = \alpha\}$. In this section we give some geometric estimates on $\hat\cV_\alpha$ outside of a neighborhood of critical circle maps.

We start with the following result that was essentially proved  in \cite{GY},  but was only formulated for the Arnold's family, see Corollary 1.9.

Denote  $\alpha_{-1} =1$, $\alpha_k = G^k(\alpha)$, where $G(x)$ is the Gauss map, $G(x)  = \{1/x\}$. Recall that the Yoccoz-Brjuno function is
$$
\Phi(\alpha) = \sum_{j=0}^{\infty} \alpha_{-1} \alpha_0  \dots \alpha_{j-1} \log \frac{1}{\alpha_j} 
$$ 
and the set of Brjuno rotation numbers $\mathcal B\subset \bbR$ consists of $\alpha\notin\bbQ$ for which  $\Phi(\alpha)$ is finite. 
For any sequence $s_n\to 0$, let $\mathcal B_{\{s_n\}}\subset \mathcal B$ be the set of numbers $\alpha=[a_1, a_2, \dots]$ such that 
$$
\sum_{j=n}^{\infty} \alpha_{-1} \alpha_0  \dots \alpha_{j-1} \log \frac{1}{\alpha_j} \le s_n.
$$
In \cite{GY}, we defined an analytic renormalization operator $$\mathcal R\colon \mathcal D_\nu \to \mathcal D_\nu$$ and proved its hyperbolicity on Brjuno rotations. For any $\alpha\in \mathcal B$, let $R_\alpha(z)=z+\alpha$ be the corresponding rotation and let $\mathcal V_\alpha\subset \mathcal D_\nu$ be the local stable manifold of $\mathcal R$ at $R_\alpha$.

\begin{lemma}
\label{lem-deriv-diff-rot}
 For any $\nu>0$ and any sequence $\{s_n\}, s_n\to  0$, there exists $\delta>0$ such that the surfaces $\mathcal V_\alpha\subset \mathcal D_\nu$ for $\alpha\in \mathcal B_{\{s_n\}}$ are graphs of analytic functions over the open ball $U_\delta(0) \subset V_0 = \{f(0)=0\}$. They consist of maps that are conjugate to $R_\alpha$ in  a strip of width $\Pi_{0.5\nu}$.

\end{lemma}
\begin{proof}
Risler's theorem implies that for any Brjuno number $\alpha$, the set  $\mathcal V_\alpha$ is a local analytic submanifold of codimension 1 at $R_\alpha$ that consists of maps that are conjugate to $R_\alpha$.


The proof of Theorem 1.5 in \cite{GY} shows that the statement of Lemma \ref{lem-deriv-diff-rot} holds if $\nu  > c_1 \Phi(\alpha)+c_2$ for large universal $c_1, c_2$. Indeed, for $\alpha = [a_1, a_2, \dots]$, the analytic surface $\cV_\alpha$ was constructed as a limit of the sequence of analytic surfaces $\cV_n$ that correspond to periodic continued fractions of the form $[a_1, a_2, \dots, a_n, a_1, a_2, \dots, a_n, \dots]$, and Lemma 5.8 shows that all these surfaces for all $\alpha\in \mathcal B_{s_n}$ are graphs over the same ball in $V_0$.  This implies that all $\cV_\alpha$ are graphs over the same ball in $V_0$ if $\nu > c_1 \Phi(\alpha)+c_2$. Moreover, they consist of maps that are conjugate to $R_\alpha$ in $\Pi_{0.7\nu}$ due to Lemma 5.5. 

The reduction to the case of an arbitrary $\nu$ is the same as in Sec. 7 of \cite{GY}. Namely, set $\alpha_k=G^k(\alpha)$; use a renormalization $\mathcal R_{\tilde \eps, m}$ of a high order $m=m(\{s_n\})$ to take the $\delta$-neighborhood of the rotation $R_\alpha$ to a $\delta'$-neighborhood of rotation $R_{\alpha_k}$ in a space $\mathcal D_{\tilde \eps}$, where  $\tilde \eps > c_1 \Phi(\alpha_k)+c_2$. Due to the above,  $\mathcal V_{\alpha_k}\cap U_{\delta'}(R_{\alpha_k})$ are analytic submanifolds that  project to one and the same ball in $V_0$ and consist of maps that are analytically conjugate to $\mathcal R_{\alpha_k}$ in $
\Pi_{0.7\tilde \eps}$. For large $\tilde \eps$, this implies that the conjugacy is close to the identity.

In \cite{GY}, we further define $\mathcal V_\alpha$ as a preimage of $\mathcal V_{\alpha_k} $ under $\mathcal R_{\tilde \eps, m}$.
The construction of renormalization implies that the maps  $f\in \mathcal V_\alpha\cap U_\delta(R_\alpha)$ for small $\delta$ are analytically conjugate to $R_\alpha$, with conjugacy defined in $\Pi_{0.5\eps} $ and close to the identity.  This estimate on the conjugacy implies a uniform  estimate on the slope of $T \mathcal V_\alpha$.  Thus $\mathcal V_\alpha$, $\alpha\in \mathcal B_{\{s_n\}}$,  project to one and the same ball in $V_0$.


\end{proof}
Roughly speaking, this lemma shows that $\cV_\alpha$, $\alpha\in \mathcal B_{\{s_n\}}$, form an analytic foliation near $R_\alpha$: the union  $(\bigcup_{\alpha \in \mathcal B_{\{s_n\}}} \cV_\alpha)$ is a union of graphs of analytic functions over the same codimension-1 ball.

Now, we prove that the same statement holds in the space of triples in a neighborhood of any triple $\hat f$ with Herman rotation number.
We use the conjugacy of $p(\hat f)$ with rotation to reduce this statement to the previous lemma.

\begin{lemma}
\label{lem-local-bound-deriv}

Consider  an $\bbR$-preserving triple $\hat f\in D$ such that $p$ is defined on its neighborhood and $p(\hat f)$ is a circle diffeomorphism with $\rot (p(\hat f)) = \alpha_0 \in \mathcal H$.

 
 Then for any $\{s_n\}, s_n\to 0$, such that $\alpha_0\in \mathcal B_{\{s_n\}}$, there exists $\delta$ such that whenever  $\alpha\in\mathcal B_{\{s_n\}}$ satisfies $|\alpha-\alpha_0|<\delta$, there exist local analytic submanifolds $\hat {\mathcal V}_\alpha \subset  D$ near $\hat f$. They consist of triples $\hat g$ such that $p(\hat g)$ is a map of an annulus that is conjugate to $R_\alpha$ in a strip of width bounded from below.

 Let $V+W$ be some splitting of $T D$ with $\dim W=1$ such that $W$ is transversal to $T\hat {\mathcal V}_{\alpha_0}$ at $\hat f$.
 Then all $\hat {\mathcal V}_\alpha$, $|\alpha-\alpha_0|<\delta$, can be represented as graphs in $\hat f + V+W$ over one and the same ball $U_{\tau}(0)$ in $V$.
 Moreover, if $\hat {\mathcal  V}_{\alpha}$ is represented as a graph $(x, Q_\alpha(x))$ in the coordinates $V, W$, where $Q_{\alpha}\colon V \to W$, then for any $n$, the derivatives  $Q_\alpha^{(n)}$ are bounded on $U_{\tau/2}(0)$ uniformly on $\alpha$.
\end{lemma}

\begin{proof}

\noindent
\textbf{Construction of $\hat  {\mathcal V}_{\alpha}$.}




Let $f=p(\hat f)$. Due to Yoccoz's theorem \cite{Yoccoz2002}, $f$ is conjugate to $R_{\alpha_0}$ via some analytic circle diffeomorphism $h$; let $L_h$ be a conjugacy with $h$, $f=L_h R_\alpha$. Fix $\mu>0$ and consider analytic codimension-1 submanifolds $\mathcal V_\alpha\subset \mathcal D_\mu$, $\alpha\in \mathcal B_{\{s_n\}}$, near $R_{\alpha_0}$. Due to the previous lemma, for some $\tau$, all $\mathcal V_\alpha$ are graphs over the $\tau$-ball $U_\tau(0)$ in $V_0$, and thus the relative boundaries of $\mathcal V_\alpha$ are uniformly detached from $R_\alpha$. Also,  all maps in $\mathcal V_\alpha \cap U_{\tau}(R_\alpha)$ are conjugate to the rotation $R_{\alpha}$ in $\Pi_{0.5\mu}$. In what follows, we trim $\mathcal V_\alpha$ to $\mathcal V_\alpha \cap U_{\tau}(R_\alpha)$.

Note that the tangent spaces of $\mathcal V_\alpha$ at rotations are parallel, given by $\int_{\bbR/\bbZ} v dx =0$. Since $\mathcal V_\alpha$ are graphs of bounded functions, their second derivatives are bounded in $U_{\tau/2}(0)$ and thus their tangent spaces $T\mathcal V_\alpha$ are close to $T\mathcal V_{\alpha_0}$ at $R_{\alpha_0}$ for $\alpha\approx \alpha_0$ and small $\tau$.

Consider $\hat {\mathcal V}_{\alpha} = p^{-1}(L_h (\mathcal V_{\alpha}))\cap U_{\delta}(\hat f)$. These are analytic submanifolds, since $p$ is analytic and  $L_h$ is an analytic invertible operator. They are non-degenerate and have codimension 1, since $dp$ does not take $TD$ into $T\mathcal V_\alpha$ (indeed, the tangent vector field to the family of triples  $(F+\eps, H_a, G)$ is a unit vector field that does not belong to $T\mathcal V_\alpha$). On  $\hat {\mathcal V}_\alpha$,  projections of all triples are conjugate to $R_{\alpha}$ on some fixed strip $\Pi_\nu\subset h^{-1}(\Pi_{0.5\mu})$ that depends on $f$ only. Finally, tangent spaces of $\hat {\mathcal V}_\alpha$ in a neighborhood of $\hat f$ are close to the tangent space of $\hat {\mathcal V}_{\alpha_0}$ at $\hat f$, since $T\mathcal V_{\alpha}$ is  close to $T|_{R_{\alpha_0}}\mathcal V_{\alpha_0}$.

\noindent

\textbf{Boundaries of $\hat {\mathcal V}_\alpha$ and the domain of $Q_\alpha$}

If $\delta$ was chosen sufficiently small, the relative boundaries of $\hat {\mathcal V}_{\alpha}$ belong to the boundary of $U_\delta(\hat f)$: indeed, the relative boundaries of $\mathcal V_\alpha$ are detached from $R_\alpha$, and thus the relative boundaries of $p^{-1}(L_h (\mathcal V_{\alpha}))$ are detached from $\hat f$.  After trimming by $U_\delta(\hat f)$ with sufficiently small $\delta$, all relative boundaries will belong to the boundary of $U_\delta(\hat f)$.

Tangent spaces to $\hat{\mathcal V}_{\alpha}$ at $\hat g\approx \hat f$ are uniformly transversal to $W$ since they are close to $T|_{\hat f} \mathcal V_{\alpha_0}$, which is transversal to $W$.   Thus for $|\alpha-\alpha_0|<\delta'$ with small $\delta'$, all $\hat{\mathcal V}_{\alpha}$ are graphs over one and the same subset of $V$, i.e. $Q_\alpha$ have same domain of definition $U_\tau(0)\subset V$.


\noindent
\textbf{Derivatives of $Q_\alpha$.}

Since the graphs of $Q_\alpha$ are located in a small neighborhood of $\hat f$, $Q_\alpha$ are bounded in $U_\tau(0)$.
 Cauchy estimates imply that on $U_{\tau/2}(0)$ all functions $Q_\alpha$, $\alpha\approx \alpha_0$, have uniformly bounded derivatives. 
\end{proof}




\section{General statement on smoothness of stable-unstable manifolds and proof of  Theorem \ref{th-smoothness}}
\label{sec-stableunstable}

In this section, we formulate a general version of the theorem on the smoothness of invariant manifolds of hyperbolic operators with a two-dimensional unstable bundle, and use this result to prove Theorem \ref{th-smoothness}. The proof of Theorem \ref{th-Banach-op} will be given in Section \ref{sec-Banach-op}.

\begin{theorem}[Smoothness of stable-unstable  invariant manifolds for operators in Banach spaces]
\label{th-Banach-op}
 Consider a real Banach space $D$, a linear closed codimension-1 space $D_0$ and open half-spaces $D_{-}$ and $D_+$ such that $D = D_0 \cup D_- \cup D_+$. Consider a sequence of smooth operators $\{R_n\}$, $$R_n \colon D\to D\text{ with }R_n D_0\subset D_0;$$ choose a point $x_0$ and set $\{x_n\}$ to be its orbit, $x_{n+1} = R_n x_n$. Fix $k\in \bbN$ and for a  certain $\delta>0$, assume  that each $R_n$ is defined and $k$ times differentiable on a semi-neighborhood $B_\delta(x_n)\cap (D_0 \cup D_-)$ and has bounded derivatives up to the order $k$, with bounds independent on $n$.

 For  a splitting $D = \xi_0 + \eta_0 + l_0$ where $$\dim \xi_0 = \dim \eta_0=1,\; l_0+\eta_0 = D_0,$$ we  set $$\xi_{n+1} = dR_n|_{x_n} \xi_n, \; \eta_{n+1} = dR_n|_{x_n} \eta_n,\; l_{n+1} = dR_n|_{x_n} l_n.$$
 Suppose that $l_n$ is a stable distribution for $\{R_n\}$ and $\xi_n+\eta_n$ is its unstable distribution.
Suppose that $\xi_n$ is uniformly transversal to $D_0$,  the minimal expansion rate $\lambda_1$ of $ \{R_n\}$ along $\eta_n$ is larger than the  maximal expansion rate $\lambda_2$ along $\xi_n$, and moreover, $$k<  \frac {\log \lambda_1}{\log \lambda_2}.$$ 

Suppose that there exist codimension-1 local manifolds $M_n$ with border  that pass through $x_n$, such that  $(R_n M_n )\cap B_{\delta}(x_0)\subset  M_{n+1}$. Suppose that $M_n\cap D_0$ is a stable foliation for $R_n$, $T(M_n\cap D_0)=l_n$. Let   $M_n$ be given by a function $$Q_n \colon  (l_n+\xi_n) \to \eta_n;$$ suppose that $Q_n$  is well-defined with uniformly bounded derivative in $ (l_n+\xi_n) \cap (D_0\cup D_-)\cap B_\delta(x_n) $ where $\delta$ does not depend on $n$. Suppose that $Q_n $ is $k$-differentiable in $ (l_n+\xi_n)\cap D_- \cap B_\delta(x_n)$, and for any $\delta_1<\delta$, all its derivatives up to the order $k$ are uniformly (on $n$) bounded on $D_-\cap (l_n+\xi_n)\cap (B_\delta(D_0)\setminus B_{\delta_1}(D_0))$.

 Then $Q_n^{(l)}(x)$ with $l\le k $ has a limit as $x\to D_0$, thus $Q_n(x)$ is $k$ times continuously differentiable.
\end{theorem}

\subsection*{Assumptions of Theorem \ref{th-Banach-op} hold under the assumptions of  Theorem \ref{th-smoothness}}

Choose the space of triples $D_{\eps, \eps'}$ that satisfies Lemma \ref{lem-distributions}, and let $D$ be the space of $\bbR$-preserving triples $D^{\bbR}_{\eps, \eps'}$. Let $D_0 = \{a=0\}$, $D_- = \{a<0\}$, $D_+ = \{a>0\}$. Set $R_n=\mathcal R$. Let $x_0=\hat g$, take $\eta_n$, $\xi_n$ to be the unstable distributions $\eta_{\hat g_n}$ and $\xi_{\hat g_n}$ respectively. Take $l_n$ to be the stable distibution of $\mathcal R$ at $\mathcal R^n (\hat g) = \hat g_n$; we will also use the notation $l_{\hat g}$ for the stable distribution of $\mathcal R$ at $\hat g$. 

Since the orbit $\{\hat g_n\}$ belongs to a compact set $$\Lambda_c = \{f\in \Lambda, \rot f\text{ is of type bounded by }c\},$$ we conclude that an analytic operator $\mathcal R$ is defined and has uniformly bounded derivatives of each order in  $B_\delta(\hat g_n) \cap (D_0\cup D_-)$ for a small $\delta$.

The inequality on $\lambda_1, \lambda_2$ and the transversality condition is included in the assumptions of Theorem \ref{th-smoothness}. It remains to establish the required properties of the invariant manifolds $$M_n = \{\hat f \in D \mid \rot p(\hat f) = \rot p(\hat g_n) \} \cap B_\delta( \hat g_n)$$ for sufficiently small $\delta$.

Recall that we denote $\mathcal V_\alpha = \{\rot f=\alpha\}\subset \mathcal D_\eps$.
We will not restrict ourselves to the sequence $\hat g_n$,  we will rather establish the required properties for all the manifolds $\hat {\mathcal V}_{\alpha} = p^{-1}(\mathcal V_\alpha)\subset D $ with $\alpha$ of type bounded by $c$. Note that  $M_n=\hat {\mathcal V}_{\rot (p(\hat g_n))} \cap B_\delta(\hat g_n)$.

   Due to  Theorem \ref{th-C1}, the manifolds $\mathcal V_\alpha$ are smooth in $\mathcal D_\eps$ up to $\mathcal D_{\eps}^{cr}$. Thus $\hat {\mathcal V}_{\alpha}$  are smooth in $D$ up to $D_0$. Theorem \ref{th-C1} also  implies that the tangent space to  $\hat {\mathcal V}_\alpha \cap D_0$ at the critical triple $\hat g\in D_0$ is the preimage under $dp$ of the stable subspace of  $\mathcal R_{cyl}\colon \mathcal D_{\eps}^{cr}\to \mathcal D_{\eps}^{cr}$, thus coincides with the stable subspace of $\mathcal R \colon D\to D$. This implies that for each $\hat g\in \Lambda$, $T_{\hat g}\hat {\mathcal V}_{\alpha}$ is transversal  to the unstable direction $\eta_{\hat g}$ of $\mathcal R$. Thus, $\hat {\mathcal V}_{\alpha}\cap B_\delta(\hat g)$ can be represented as a graph $(x,Q_{\hat g}(x))$ of a smooth function $Q_{\hat g}\colon (l_{\hat g } + \xi_{\hat g})\to \eta_{\hat g}$ defined in some neighborhood of $0$. 

Let us prove that all $Q_{\hat g}$ are defined in a $\delta'$-semi-neighborhood of $\hat g\in \Lambda_c$ for a certain $\delta'$, and have uniformly bounded derivatives.  
 
Recall that for bounded-type $\alpha$,  $\hat {\mathcal V}_\alpha \cap D_-$ is an analytic codimension-1 manifold at each its point, due to Risler's theorem and analyticity of $p$. Also, $\hat {\mathcal V}_\alpha \cap D_0$ is an analytic codimension-1 manifold at each its point, namely a leaf of the stable foliation of $\mathcal R_{cyl}$ (see \cite[Theorem 3.8]{Ya4}). Thus the boundary of the domain of the corresponding function $Q_{\hat g}$ consists of points $x$ such that either $(x, Q_{\hat g}(x))\in \partial B_\delta(\hat g)$, or  $\|dQ_{\hat g}\|=\infty$. 
 The next lemma shows that $Q_{\hat g}$ have bounded derivatives on a certain small semi-neighborhood of $D_0$.  Thus $Q_{\hat g}$ are defined on one and the same neighborhood of $\hat g$.

\begin{lemma}
 For each $c$, there exists $\delta$ and $K$ such that  if  $\hat g\in \Lambda_c$, if  $$(x, Q_{\hat g}(x)) \subset B_\delta(\hat g),$$ then  $\|dQ_{\hat g}|_{x}\|<K$.
\end{lemma}
\begin{proof}
 Suppose that for some $K$, this statement fails for any $\delta$. Due to the compactness of $\Lambda_c$, there exists a sequence of critical triples $\hat g_{k} \to \hat g\in \Lambda_c$, non-critical triples $\hat h_k\to \hat g$ with the same rotation numbers $\rho_k$, and vector fields $v_k\in l_k+ \xi_k\subset T|_{\hat g_k}D$  such that $\|v_k\|=1$, and the differentials $dQ_{\hat g_k}$ at $\hat h_k$ satisfy $dQ_{\hat g_k} v_k=C_k \eta_k$ with $|C_k| \ge K$. Due to the definition of $Q_{\hat g_k}$, the vector field  $v_k + dQ_{\hat g_k} v_k$ is tangent to the manifold $\{\rot p(\hat f) = \rot p(\hat g_k)\}$ at $\hat h_k$.

 Let $g=p(\hat g), g_k=p(\hat g_k), h_k=p(\hat h_k)$ be the corresponding circle maps.
 Let $\mu_k$, $\tilde \mu_k$, and $\mu$  be the (-1)-measures for $h_k, g_k, g$.  Then   $\mu_k\to \mu$ and $\tilde \mu_k \to \mu$ weakly due to Theorems \ref{th-measure} and \ref{th-C1}.
 
 Since $v_k+dQ_{\rho_k} v_k$ is tangent to $\{\rot p(\hat f) = \text{const}\}$ at $\hat h_k$,  we have
 \begin{equation}
 \label{eq-v-1}
 \int (dp)v_k|_{h_k^{-1}} d\mu_k +\int (dp)(dQ_{\rho_k} v_k)|_{h_k^{-1}} d\mu_k  =0.
 \end{equation} 
 The first summand is uniformly bounded since $\|v_k\|=1$, $\|dp\| $ is uniformly bounded on a compact set $\Lambda_c$, and $\mu_k$ is a probability measure. Thus  $\int (dp)dQ_k v_k|_{h_k^{-1}} d\mu_k $ is uniformly bounded. On the other hand, $dQ_k v_k = C_k \eta_k$ with $|C_k|\ge K$, and $ (dp)\eta_k = \eta_{g_k}$ is the unstable direction for $\mathcal R_{cyl} $ at $g_k$.   Since $\eta_{g_k}$ is uniformly transversal to the stable distribution, the values of the linear functional $L_{g_k} $ on $\eta_k$ are bounded away from zero, i.e. $|\int \eta_k|_{g_k^{-1}} d\tilde \mu_k |$ is bounded away from zero.  Since $\tilde \mu_k\to\mu$, $\mu_k\to\mu$, this implies an upper bound on $K$. 
 
 Thus the statement of the lemma holds for sufficiently large $K$. 
\end{proof}
\begin{remark}
One can prove that the statement holds for any $K$, i.e. the first derivatives of $Q_{\hat f}$ uniformly tend to zero as $(x, Q_{\hat f}(x))\to D_0$. Indeed, due to Lemma \ref{lem-weakstable}, $l_{k}+\xi_{k}$ coincides with the subspace given by $\int (dp)v|_{g_k^{-1}} \tilde d\mu_{k}=0  $, and since $v_k\in l_k+\xi_k$ and both $g_k$ and $h_k$ tends to $g$, we have $\lim_{k\to \infty } \int (dp)v_k|_{h_k^{-1}} d\mu_k =  \lim_{k\to\infty} \int (dp)v_k|_{g_k^{-1}} d\tilde \mu_{k} =0$. Since  the first summand in \eqref{eq-v-1} tends to zero, the second summand also tends to zero,  and thus $|C_k|>K$ is not possible for any $K$. 
\end{remark}

For negative $a_1, a_0$, let the set $K_{a_0, a_1}$ of triples $(F, H_a, G)$ be given by $a_1< a<a_0$.
It remains to prove the uniform estimate on higher derivatives of $Q_n$ on the set $D_-\cap (l_n+\xi_n)\cap B_\delta(D_0)\setminus B_{\delta_1}(D_0)$, i.e. in $K_{a_0, a_1}$ with small $a_0, a_1$.

The following lemma completes the reduction.
\begin{lemma}
\label{lem-qQ}
For any $0>a_0>a_1$, the  functions $Q_n\colon (l_n+\xi_n)\to \bbR$ have uniformly bounded derivatives on $ (l_n+\xi_n)\cap K_{a_0, a_1}$.
\end{lemma}
\begin{proof}
Assume the contrary: suppose that derivatives $Q_k^{(s)}$ have unbounded norms at some points $\hat h_k$. Consider a larger space $D_{\tilde \eps, \tilde \eps'}$ of triples with slightly smaller domain of definition $\tilde \eps<\eps, \tilde \eps'<\eps'$. Clearly, the corresponding derivatives still have unbounded norms. Extracting a convergent subsequence from $\hat h_k$ in this new space,  we get the limit $\hat h=\lim_{k\to \infty} \hat h_{n_k}\in D_{\tilde \eps, \tilde \eps'} \cap K_{[a_0, a_1]}$. Note that  $\rot p(\hat h) =\lim \alpha_{n_k}$ has bounded type. Consider the splitting $T D = V+W$ where $V$ is the limit of $l_{n_k}+\xi_{n_k}$ and $W$ is the limit of $\eta_{n_k}$. Let $q_{n_k}\colon V\to W$ be  the functions whose graphs coincide with $M_{n_k}$; they also have unbounded derivatives of order $s$, since $l_{n_k}+\xi_{n_k} \to V$ and $\eta_{n_k}\to W$.

This contradicts Lemma \ref{lem-local-bound-deriv}, applied to the triple $\hat h\in D_{\tilde \eps, \tilde \eps'}$.  
\end{proof}

\subsection*{Assertions of Theorem \ref{th-smoothness} follow from the assertions of Theorem \ref{th-Banach-op}}

We will only use the fact that all vector fields $d  f_\mu/d\mu_2$ for $\mu_1=0$ are transversal to the surface $\{\rot f = \text{const}\}$. This clearly follows from the assumptions, since the positive vector field  $v=df_{\mu}/d\mu_2>0$ cannot satisfy $\int v|_{f^{-1}(z)}d \mu_f =0$ for a (-1) -measure $\mu_f$  and thus cannot be tangent to the surface $\{\rot f=\text{const}\}$.  

First, let us reduce the general statement to the case when the critical map  $f_0$ is close to $g\in \mathcal I$. Indeed, results of \cite{Ya3} imply that $\tilde f_\mu = \mathcal R_{cyl}^n f_\mu$ is well-defined and close to $\mathcal I$. In brief, let $\mathcal R_{\text{pairs}}$ be the renormalization of  commuting pairs. Then $\mathcal R_{\text{pairs}}^n(f_\mu)$ is a rescaled pair $(f^{q_n}_\mu, f^{q_{n+1}}_\mu)$. Due to \cite{Ya3}, for $\mu=0$, for large $n$, this pair is close to a commuting pair from the attractor of $\mathcal R_{\text{pairs}}$. In particular, $f^{q_n}_\mu$ with small $\mu$ admits a fundamental crescent that joins hyperbolic repelling points of $f^{q_n}_\mu$, and the corresponding linearizing chart $\Psi$ is defined. So $\tilde f_\mu = \cren^n f_\mu$ is well-defined and close to $\mathcal I$.

Replacing $f_\mu$ with $\tilde f_\mu$, we may and will assume that $f_{\mu}$ is close to $g\in \mathcal I$ and thus $f_0$ belongs to the stable manifold  $W_g^s$ of $g$ for $\cren$. 

The transversality condition is preserved since $\cren$ preserves the foliation $\{\rot f =\text{const}\}$ in $D_0$ and has a transversal unstable direction. Also, the condition $\rot f_\mu = \rot g$ is equivalent to $\rot \tilde f_\mu = \rot \cren^n g$ for $\mu\approx 0$, hence we do not change the Arnold tongue when we replace $f_\mu$ with $\tilde f_\mu$. 

Lemma \ref{lem-lift} provides us with an analytic family  $\hat f_\mu\subset D$ such that $p(\hat f_\mu) = h_\mu f_\mu h_{\mu}^{-1}$; for $\mu_1=0$ (i.e. when $f_\mu$ is critical), we have $h_\mu=id$ and $\hat f_\mu=j(f_\mu)$. Note that  the vector field $d(h_\mu f_\mu h_\mu^{-1})/d\mu_2$ is  still transversal to the surface $\{\rot f=\alpha\}$ since the conjugacy is an invertible analytic operator that takes tangent spaces to these surfaces to tangent spaces, thus the condition of transversality is preserved.
Thus  $d \hat f_\mu/d\mu_2 \in TD$  is transversal to the surface given by $$\{\rot (p(F, H_a, G))=\alpha\},$$ since it projects under $d p$ to a vector  $d(h_\mu f_\mu h_\mu^{-1})/d\mu_2$ that is transversal to $\{\rot f=\alpha\}$.



Since $f_0\in W^s_g$ for $\cren$, the construction of the stable foliation of $\mathcal R$ (Lemma \ref{lem-distributions}) implies that $j(f_0)=\hat f_0$ belongs to the stable manifold for $\mathcal R$ that contains $\hat g$, thus  for any $\delta$, we can find $n$ such that  $\mathcal R^n \hat f_0$ is $\delta$-close to $\mathcal R^n \hat g\in \Lambda$.
  Note that   $d (\mathcal R^n \hat f_\mu)/d\mu_2 \in TD$  is still  transversal to the surface  $$\{\rot (p(F, H_a, G))=\rot p(\mathcal R^n \hat g)\}$$ for $\mu_1=0$ since $\mathcal R^n$ preserves transversality to its invariant manifolds.



Theorem \ref{th-Banach-op} implies that the function $Q_0$ that defines the surface $\{(F, H_a, G) \in D \mid \rot p(F, H_a, G)=\rot p(\mathcal R^n \hat g)\}$ is $k$ times continuously differentiable at $\mathcal R^n (\hat f_\mu)$ for $\mu_1=0$ and small $\mu_2, \dots, \mu_n$. 


The Implicit Function Theorem, applied to $Q_0$ and $\mu \mapsto (\mathcal R^n \hat f_\mu)$, implies that the function $\mu_2(\mu_1, \mu_3\dots, \mu_n)$ defined implicitly by $$\{\mu\mid \rot p(\mathcal R^n \hat f_\mu)=\rot p(\mathcal R^n \hat g)\}$$ is $k$ times continuously differentiable at any point $(0,\mu_3, \dots, \mu_n)$ sufficiently close to zero. Since $\rot p(\mathcal R^n \hat f_\mu)=\rot p(\mathcal R^n \hat g)$ is equivalent to  $\rot f_\mu = \rot p(\hat f_\mu)=\rot (\hat g)=\alpha$ in a small neighborhood of $\mu=0$, this function defines the Arnold $\alpha$-tongue in the initial family $f_\mu$, which completes the proof.

\section{Proof of Theorem \ref{th-Banach-op} on the smoothness of stable-unstable manifolds.}
\label{sec-Banach-op}
\begin{proof}[Proof of Theorem \ref{th-Banach-op}]
  $\;$

  \smallskip
  \noindent
\textbf{Step 1. Coordinates on $TD$ and notation.}

We are going to make an appropriate choice of the vectors  $\tl\xi_n$, $\tl\eta_n$ which generate the direction fields $ \xi_n, \eta_n$. 
Namely, for a small value $\eps>0$ to be fixed later (in Lemma \ref{lem-DnBn}), we choose $\tl\eta_n$ so that  $$d |_{x_n} R_n \tl\eta_{n} =k \tl\eta_{n+1}\text{ with }|k|\ge \lambda_1-\eps$$
 and  $\|\tl \eta_n\|$ are bounded away from zero and infinity.  Hence the ratio of the standard norm and the coordinate norm on the one-dimensional subspace $\eta_n$ that corresponds to the basis $\tl\eta_n$ is bounded. 
 This choice is possible due to the fact that  $\lambda_1$ is the minimal expansion rate along $\eta_n$.  

Set $L_n = \xi_n  + l_n $. Similarly, we  will choose vectors  $\tl\xi_n$ that generate the spaces $\xi_n$, and the norms on $l_n$ uniformly equivalent to the standard norm,  so that $$\|(d|_{x_n} R_n)|_{L_n}\| \le \lambda_2+\eps$$ with respect to the induced norm on $L_n =  \xi_n  + l_n $. The ratio of the standard norm and this norm on $L_n$ is bounded. The choice is possible  due to the fact that  $\lambda_2>1$ is the maximal expansion rate along $\xi_n$ and $l_n$ is stable.

Let $\pi_s^n$ be the projection onto $L_n$ along $\tl\eta_n$ and let $\pi_u^n$ be the projection onto $\tl\eta_n$ along $L_n$.
Let $d_u$ be the derivative along $\tl\eta_n$ and $d_s$ be the derivative along $L_n$.
We will use the representation $T_zD =  \eta_n + L_n$ for tangent vectors at each point $z\in M_n$, and we will also use the notation  $(a,b) $ for a point $x_n + a + b\tl\eta_n$ in a neighborhood of $x_n$, where $a\in L_n$, $b\in \bbR$.

Recall that  $Q_n\colon L_n\to \bbR$ is the function defined on a neighborhood of zero in $L_n$ such that its graph $(y, Q_n(y))$  coincides with $M_n$.

\medskip
\noindent
\textbf{Step 2: Choosing neighborhoods.}

\begin{lemma}
\label{lem-estim}
For any $\eps>0$, there exists $\mu>0$ such that for all $n$, in a $\mu$-neighborhood $U_n$ of zero  in $L_n$, for  $y\in U_n$, $z=(y, Q_n(y))$, we have:
\begin{enumerate}
\item $ \|d_s \pi_s R_n|_{z}\|<\lambda_2+2\eps $;

 \item  $\|d_u(\pi_s R_n)|_{z}\|<\eps$; $\|d_s( \pi_u R_n)|_{z} \|<\eps$;
 \item   $d\pi_u R_n|_{z} \tl\eta_n = \nu_n \tl\eta_{n+1}$ where $|\nu_n| > \lambda_1-2\eps$.
\end{enumerate}
\end{lemma}
\begin{proof}
With the choice of $\tl\eta_n$ as above, the matrix of $dR_n|_{x_n}$ in $(\tl\eta_n, L_n)$ coordinates in the domain and $(\tl\eta_{n+1}, L_{n+1})$ coordinates in the image is block-diagonal. In particular,  $$d_s \pi_u R_n|_{x_n} =0,\; d_u(\pi_s R_n)|_{x_n} =0,\; \|d_s (\pi_s R|_{x_n}) \| \le \lambda_2+\eps,$$ and $$d_u (\pi_u R|_{x_n}) \tl\eta_n = \nu_n \tl\eta_{n+1},\;|\nu_n| \ge \lambda_1-\eps.$$ Since we have uniform estimates on derivatives of $R_n$ on a $\delta$-neighborhood of $\{x_n\}$, the required estimates hold in $\mu$-neighborhoods of $x_n$ for a certain $\mu$ independent of $n$.

\end{proof}


\noindent
\textbf{Step 3. Recurrent relation between higher-order derivatives along orbits of $ R_n$.}

We have for $y\in U_n$ such that $R_n(y, Q_n(y))\in M_{n+1}$:
$$ \pi_u R_n(y, Q_n(y)) =  Q_{n+1} (\pi_s R_n(y, Q_n(y)))$$
since  a point $R_n(y, Q_n(y)) =(\pi_s R_n(y, Q_n(y)),  \pi_u R_n(y, Q_n(y)))$ belongs to $M_{n+1}=\{(y, Q_{n+1}(y))\}$.

Now, we differentiate this equality $m\le k$ times with respect to $y$ and separate away the terms that involve $m$-th derivatives of $Q_n$ and $Q_{n+1}$. We get:
\begin{multline}
\label{eq-R}
d_u(\pi_u R_n)Q^{(m)}_n(y) + \left[ d_u d_s \pi_u R_n Q^{(m-1)}_n(y)+\dots +d_s^{(m)} \pi_u R_n \right] = \\= Q_{n+1}^{(m)}|_{ (\pi_s R_n(y, Q_n(y)))} \cdot (d \, \pi_s R_n(y, Q_n(y)))^m  +  dQ_{n+1} d_u(\pi_s R_n)Q_n^{(m)}(y) +\\+ \left[ dQ_{n+1} d_u d_s \pi_s R_n Q_n^{(m-1)}(y))) + \dots\right].
\end{multline}
Here $Q^{(m)}$ is an $m$-linear form on $L_n$, it is defined on $m$-tuples of vectors from $L_n$, and $ {(d \, \pi_s R_n(y, Q_n(y)))^m}$ applies the operator  $d \, \pi_s R_n(y, Q_n(y))$ to each vector of the $m$-tuple.

This relates $Q^{(m)}_n(y)$ to  $Q^{(m)}_{n+1}$ at a point $\pi_s R_n(y, Q_n(y))$.

We inroduce the following notation:
\begin{itemize}
 \item $D_n$ acts on tuples of $m$ vectors in $TD$ by
 $$D_n :=  {(d \, \pi_s R_n (y, Q_n(y)))^m}$$
(this operator takes a tuple of $m$ vectors $v_1,\dots, v_m$ in $L_n$ to the tuple $d \, \pi_s R_n(y, Q_n(y))) v_1$, $d \, \pi_s R_n(y, Q_n(y))) v_2,$ $\dots $, where all derivatives are computed at $(y, Q_n(y))$),
\item  $ B_n\colon \tl \eta_n  \to  \tl\eta_{n+1} $ is given by $$B_n:= - dQ_{n+1}  d_u(\pi_s R_n)+   d_u \pi_u R_n$$ at $(y, Q_n(y))$;
\item $A_n$ is the sum of all components in the square brackets in \eqref{eq-R}; this sum involves derivatives of $ R_n$ and lower derivatives of $Q_n$, $Q_{n+1}$.
\end{itemize}

Now,  \eqref{eq-R} turns into the relation
  $$B_n Q^{(m)}_{n} (y)  =  Q^{(m)}_{n+1}|_{\pi_s R_n(y, Q_n(y)) } D_n +A_n.$$


\noindent
\textbf{Step 4. Inverting $B_n$ and iterating the recurrent relation.}

The proof of Theorem \ref{th-smoothness} is by induction on $k$. Namely, we will prove that derivatives $Q_m^{(s)}$ have limits as $(y, Q_m(y))$ tends to $x\in D_0$ and are bounded on $M_n$ uniformly on $n$. The base $k=1$ is included in the assertions of the theorem. Suppose that derivatives of $Q_n$ up to order $m-1$ are bounded uniformly on $n$ and have limits as $(y, Q_m(y))\to x\in D_0$.
\begin{lemma}
\label{lem-An}
 For sufficiently small $\eps>0$,  the terms $A_n$ are bounded by the same constant for all $n$ and have limits as $(y, Q_n(y))\in M_n$ tends to $x\in D_0$.
\end{lemma}
\begin{proof}
 This holds since $A_n$ is a combination of lower derivatives of $Q_n$ (that are uniformly bounded and have limits due to the inductive statement) and derivatives of $R_n$ that are continuous and  uniformly bounded in a $\delta$-neighborhood of $\{x_n\}$.
\end{proof}

The constant $\eps>0$ will be chosen to satisfy the following. 

\begin{lemma}
\label{lem-DnBn}
 Suppose that the norms of $dQ_n$ in $U_n$ are bounded by $K$.  Then   $\|D_n\|\le (\lambda_2+\eps+2K\eps)^m$ and $\|B_n-\nu_n\|< 2K\eps$ for some  $|\nu_n| > \lambda_1-2\eps$.

 For sufficiently small $\eps$, $B_n$ is invertible and $\|D_n\|\cdot \|B_n^{-1}\|<1$.
\end{lemma}
\begin{proof}
The first statement follows from Lemma \ref{lem-estim} since $$d \, \pi_s R_n(y, Q_n(y))) = d_s \pi_s R_n + d_u \pi_s R_n \cdot dQ_n$$ and its norm is thus bounded by $\lambda_2+\eps+2K\eps$.

The second estimate is a direct corollary of Lemma \ref{lem-estim} and a bound on $dQ_n$.

For small $\eps$,  $B_n$ is invertible and $\|B_n^{-1}\|< (\lambda_1-2\eps-2K\eps)^{-1}$, thus $\|D_n\|\cdot \|B_n^{-1}\|<1$ for small $\eps$ since $\lambda_2^m < \lambda_1$.
\end{proof}

We get
$$Q^{(m)}_{n} (y)  =  B_n^{-1} Q^{(m)}_{n+1}|_{ (\pi_s R_n(y, Q_n(y)))} D_n +B_n^{-1}A_n.$$



Suppose that $L_n$-projections of images of the point $(y, Q_0(y))$ under $ R_1, R_2R_1, \dots, R_{n}\dots R_2 R_1$ stay in the domains $U_n$. Iterating the relation between $Q^{(m)}_n$ and $Q^{(m)}_{n+1}$, we get
\begin{multline}
\label{eq-recur}
 Q^{(m)}_{0} (y)  =  B_1^{-1} B_2^{-1} \dots B_n^{-1} Q^{(m)}_{n+1}D_n\dots  D_2 D_1 + \\+ \dots + B_1^{-1}B_2^{-1}A_2 D_1+  B_1^{-1}A_1
\end{multline}
where $Q^{(m)}_{n+1}$ is computed at a point $\pi_s R_nR_{n-1}\dots R_1(y, Q_{0}(y))$.

\medskip
\noindent
\textbf{Step 5. End of the proof.}

\begin{lemma}
\label{lem-delta1}
For sufficiently small $\delta>0$, there exists $\delta_1>0$, $\delta_1<\delta$, with the following property. For any $x\in M_0\cap D_0$ that belongs to $B_{\delta/2}(x_0)$, for any  $z\in M_0\cap D_-$ with $\dist(z, x)<\delta_1$, there exists $N$ such that $N$ iterates of $z$ under $R_1, R_2R_1, \dots$ stay in neighborhoods $B_\delta(x_k)$, and $R_N\dots R_1(z) $ is outside $B_{\delta_1}(D_0)$.
\end{lemma}
\begin{proof}
Select a point $w\in M_0\cap D_0$ such that $(z-w)\in \xi_0+\eta_0$; this is possible since $TM_0\cap D_0 = l_0$ is transversal to $\xi_0+\eta_0$ at $x_0$, thus in its neighborhood. 

Let $z-w = a \tl \xi_0+b\tl\eta_0$, then $\dist (z, D_0)=a$ in our metric. 
Since  $\|dQ_0\|$ is bounded, the fraction $b:a$ is uniformly bounded for all points $\in B_\delta(x_0)$ and thus the fraction $\dist(z, w): \dist(z, D_0)$ is bounded by some constant $C$ that does not depend on $z$. 

Since $\xi_n$  is an unstable distribution for $\{R_n\}$,  we have $$\| dR_{n}\dots R_1  (z-w)\| > c(1+\rho)^n \dist (z, D_0) > c_1 (1+\rho)^n \dist (z, w)$$ for some $\rho>0$. Since $\|dR_k\|$ are uniformly bounded, we have  $$\|d R_n\dots R_1 (z-w)\|< A^n \|z-w\|.$$
Due to the uniform estimate on $R''_n$, for sufficiently small $\delta$, this implies $$\dist (R_k \dots R_1z, D_0)> \tilde c (1+\rho/2)^n \dist (z,w)$$ and   $$ \dist ( R_n\dots R_1 z, R_n\dots R_1 w) < (2A)^n \dist (z, w)  $$  while $R_n\dots R_1 z$ stays in $B_{\delta}(x_n)$.

Since $ w$ is close to $x_0$ and belongs to the stable manifold $M_0\cap D_0$ of $\{R_n\}$, we have $\dist (R_n \dots R_1 w, x_n) \to 0 $; since $x\in B_{\delta/2}(x_0)$ and $z\in B_{\delta_1}(x)$, we can choose  $\delta_1$ so that the  future orbit of $w$   stays in $B_{3\delta/4}(x_n) $. 

Now, at least $K = \lfloor \log_{2A} \frac{\delta/4}{\dist(z, w)} \rfloor $ images of $z $ under $\{R_n\}$ stay in $B_{\delta/4}(R_n\dots R_1 w)$ and thus in $B_{\delta}(x_n)$. Also, the iterate of order  $N= \lfloor \log_{1+\rho/2} \frac{\delta_1}{\tilde c \cdot  \dist (z, w)}\rfloor +1$ is outside $B_{\delta_1}(D_0) $; note that $N\ge 1$ since $z$ itself belongs to $B_{\delta_1}(D_0)$.  Choose  $\delta_1>0$ such that $ \log_{2A}(\delta/4)>2\log_{1+\rho/2}(\delta_1/\tilde c)$. Then we have $K\ge 2(N-1)\ge N$. Finally, $N $ iterates of $z$  belong to $B_{\delta}(x_n)$, while the $N$-th iterate is outside $B_{\delta_1}(D_0)$. This completes the proof.

\end{proof}


Consider $(y_s, Q_0(y_s))\to x\in D_0\cap M_0$, and apply this lemma to  $z=(y_s, Q_0(y_s))$. Fix the corresponding $\delta_1$, and recall that $Q^{(m)}_n$ are uniformly bounded on   $B_{\delta}(D_0)\setminus B_{\delta_1}(D_0)$.  
For  $y=y_s$, use the formula \eqref{eq-recur} with the number of iterates $n+1=N$ provided by Lemma \ref{lem-delta1}; note that $N$ depends on $s$ and tends to infinity as $s\to \infty$. The  relation \eqref{eq-recur} applies since iterates of $(y_s, Q_0(y_s))$ stay in $B_\delta(x_n)$. Due to the choice of $N$, we have  $$\pi_s R_nR_{n-1}\dots R_1(y_s, Q_{0}(y_s))\in B_{\delta}(D_0)\setminus B_{\delta_1}(D_0).$$ Thus the first term in \eqref{eq-recur} tends to zero as $s\to \infty$ since $Q^{(m)}_{n+1}$ is bounded  and $\|D_n\| \cdot \|B_n^{-1}\|<1$ (Lemma \ref{lem-DnBn}).

The remaining terms in  \eqref{eq-recur} decrease at a uniform geometric rate  because all $A_n$-s are bounded by the same constant (Lemma \ref{lem-An}) and $\|B_k^{-1}\| \cdot \|D_k\|<1$. Each term has a limit as $s\to \infty$. Thus the sum of these terms converges as $s\to \infty$  to the sum of limits of all  the terms.  

We conclude
that as $n\to \infty$, the sequence $Q^{(m)}_{0} (y_s)$ has a limit. Since arguments work for any $m\le k$, we conclude that  $M_0$ is $k$-smooth.
Repeating the same arguments for $(y_s, Q_n(y_s)) \to x_n$, we show that all $Q_n^{(m)}$ are $k$-smooth. 

To complete the induction step, it remains to provide a uniform bound on $Q_n^{(m)}$ in small neighborhoods of $x_n$ with size independent on $n$. The above arguments work in the $\delta_1$-neighborhood of any point $x\in B_{\delta/2}(x_n)\cap D_0\cap M_n$.  Both the first term of \eqref{eq-recur} and the sum of the remaining terms are uniformly bounded due to the same estimates as above, thus $Q_n^{(m)}$ are uniformly bounded in  $\delta_1$-neighborhoods of $D_0\cap B_{\delta/2}(x_n)\cap M_n$. Also, outside $B_{\delta_1}(D_0)$, the derivatives $Q_n^{(m)}$ are uniformly bounded due to the assumptions. Thus they are uniformly bounded in $ B_{\delta/2}(x_n)$. This completes the proof.

\end{proof}

\section{Uniform hyperbolicity and expansion rates of the renormalization operator}
\label{sec-unifhyp}
In this section we present a new construction of the expanding direction of renormalization, different
from the approach taken in \cite{Ya3} and \cite{GorYa}. The proof we present uses (-1)-measures constructed in \S~\ref{sec-measure}, and 
provides explicit estimates on the expansion rate of the renormalization operator as needed for Proposition \ref{prop-exponents} and Theorem \ref{th-smoothness-htype2}.

Recall that the cylinder renormalization $\cren f$  of a critical circle map $f$ is the first return map to a fundamental domain of the map  $f^{q_{N}}$, in the straightening chart.
In this section, it will be convenient for us to use $n=N$ as a parameter; we will write $\crenn f$ to denote the cylinder renormalization on the fundamental crescent of the iterate $f^{q_n}$. 
Here and below we assume that $n$  is even; for odd $n$, all proofs are analogous, but the straightening map reverses the orientation on the circle, etc.

Recall that  $\Psi=\Psi_n$ is the straightening coordinate  used in the cylinder renormalization: $\Psi$  is defined $C\cup f^{q_n}(C)$ where $C$ is a crescent-shaped fundamental domain of $f^{q_n}$ joining fixed points of $f^{q_n}$, and $\Psi$ conjugates $f^{q_n}$ to $z\to z-1$ in $C\cup f^{q_n}(C)$. We may and will assume that $f^{q_n}(C)$ contains $[f^{q_n+q_{n+1}}(0), f^{q_{n+1}}(0)]$, and $\Psi(0)=0$.

Take a vector field $v\in T\mathcal D^{cr}_{\eps}$. Consider the family $f_a=f+av$. Let $P_a$ be the first-return map to $[f_a^{q_n+q_{n+1}}, f^{q_{n+1}}_a(0)]$. The map $\mathcal R_{cyl} f_a$ for small $a$ coincides with $\Psi_a P_a \Psi_a^{-1}$, where $\Psi_a$ is the straightening chart for $f_a^{q_n}$ as above.
 We have
 \begin{multline}\label{eq-iter}d|_f\mathcal R_{cyl, n} v =\\ \frac d{da} \mathcal R_{cyl, n} f_a =  \Psi'_a  |_{P\Psi^{-1}} + \Psi'|_{P\Psi^{-1}} \cdot P'_a|_{\Psi^{-1}(x)}  - (\Psi P \Psi^{-1})'(x)  \cdot\Psi'_a |_{\Psi^{-1}(x)} = \\ \Psi'_a  |_{\Psi^{-1}(\mathcal R_{cyl}f(x)) }   - (\mathcal R_{cyl}f )' (x)\cdot  \Psi'_a |_{\Psi^{-1}(x)}+ \Psi'|_{P\Psi^{-1}(x)} \cdot P'_a|_{\Psi^{-1}(x)}.
   \end{multline}
  In the first summand, we choose the representative of  $(\mathcal R_{cyl}f)(x)\in \mathcal \bbR/\bbZ$ in $\bbR$ that belongs to  $\Psi([f^{q_n+q_{n+1}}(0), f^{q_{n+1}}(0)])$.

Since the map $P$ equals either $f^{q_{n+1}}$ or $f^{q_{n+1}+q_n}$, we have the following two expressions for  $P'_a|_{\Psi^{-1}(x)}.$  Let $y=\Psi^{-1}(x)$.
\begin{itemize}
\item On the arc of the circle where $P(y)  = f^{q_{n+1}}(y)$, we have
\begin{equation}
 \label{eq-sum1}
 P'_a(y)  =  \sum_{l=1}^{q_{n+1}} (f^{q_{n+1}-l})'(y_{l}) v(y_{l-1}),
\end{equation}
 where $y_l = f^l (y)$;
 \item On the arc of the circle where $P(y)  = f^{q_n+q_{n+1}}(y)$, we have    \begin{equation}
\label{eq-sum2}
P'_a(y) = \sum_{l=1}^{q_n+q_{n+1}} (f^{q_n+q_{n+1}-l})'(y_{l}) v(y_{l-1}).
\end{equation}
\end{itemize}

Recall that  (-1)-measure $\mu_f$ of a smooth circle map $f$ is a probability measure such that for any continuous test function $\phi$, we have
\begin{equation}
 \label{eq-measure}
 \int \phi(x) d\mu_{f} = \int \left. [f'\phi ]\right|_{f^{-1}(x)} d\mu_{f}.
\end{equation}
 In Section \ref{sec-measure}, we proved that any cubic critical circle map has a unique  (-1)-measure $\mu_f$. Also, we proved that the stable distribution for $\cren$ is given by the condition $L_f(v)=0$ where linear functionals $L_f$ on $T\mathcal D^{cr}_\eps$ are given by  $ L_f v = \int v|_{f^{-1}(z)} d\mu_f$. Here and below, integrals are computed along the circle $\bbR/\bbZ$.
 We will prove the following.

\begin{lemma}
\label{lem-mu-zero}
The sum of the first two summands $$\Xi(x) = \Psi'_a  |_{\Psi^{-1}(\crenn f(x)) }   - (\crenn f )' (x)\cdot \Psi'_a |_{\Psi^{-1}(x)}$$ in \eqref{eq-iter} satisfies $$ L_{\crenn f} (\Xi) = \int  \Xi|_{(\crenn f)^{-1}(x)} d\mu_{\crenn f} =0.$$

\end{lemma}
\begin{proof}


%

 Recall that by Lemma \ref{lem-noatoms}, the (-1)-measure $\mu_{f}$ has no atoms.
 Thus  the formula \eqref{eq-measure} holds true for any function $\phi$ that has a single jump discontinuity.
Let $\phi =\Psi'_a |_{\Psi^{-1}(x)} $; this function is continuous on $$I=\Psi([f^{q_n+q_{n+1}}(0), f^{q_{n+1}}(0)]) = [(\crenn f)(0)-1, (\crenn f)(0)] $$ and has a single jump discontinuity at $\crenn f(0)\in \bbR/\bbZ$.
Let $x\in I$. We will write $(\crenn f)^{-1}(x)\in \mathbb R$ for the lift of the point     $(\crenn f)^{-1}(x)\in \mathbb R/\mathbb Z$ that belongs to $I$.
Then the first summand of $\Xi$, computed at a point $(\crenn f)^{-1}(x)$, equals $\phi(x)$. Also, the second summand of $\Xi$, computed at a point $(\crenn f)^{-1}(x)$, is  $\left. [(\crenn f)' \cdot   \phi ]\right|_{(\crenn f)^{-1}(x)}. $ We get
\begin{multline*}\int \Xi|_{(\crenn f)^{-1}(x)} d\mu_{\crenn f} = \\ \int  \phi(x) d\mu_{\crenn f}- \int \left. [ (\crenn f)' \phi ]\right|_{(\crenn f)^{-1}(x)}d\mu_{\crenn f} =0\end{multline*}
due to \eqref{eq-measure} applied to $\mu_{\crenn f}$.
\end{proof}

%
%
%
%
%
%

Now we will estimate the value of the linear functional $L_{\mathcal R_{cyl}f}$ on $d\mathcal R_{cyl} v$ for the unit vector field $v=1$, thus the expansion rate of the renormalization operator.
\begin{theorem}
\label{th-dR-estim}
 For any map $f\in \mathcal I$, we have an estimate
 \begin{equation}
\label{eq-dR-estim}
 L_{\crenn f}( d|_f\crenn \, 1 ) \ge \frac{c}{|J_{n}|}
 \end{equation}
where $c$ is a universal constant,  and $J_n$ is the shortest interval of the partition $\mathcal P_n$ that has the form $J_n = f^l(I_n)$, $0<l<q_{n+1}$.
 \end{theorem}
\begin{proof}
 Lemma \ref{lem-mu-zero} implies that $ L_{\mathcal R_{cyl, n}f} (\Xi)=0$; it remains to estimate the integral of the last summand in \eqref{eq-iter}. Since the interval  $ [f^{q_{n+1}+q_n}(0), f^{q_{n+1}}(0)]$ covers $I_{n+1} $ and is contained in $I_n\cup I_{n+1}$, Theorem \ref{th-realbounds}(a) implies that it is commensurable with $I_n$. The map $\Psi$ takes this interval to the inverval $[\crenn f(0)-1, \mathcal R_{cyl}f(0)] $ of length one; since $f(0)\neq 0$ on $\overline {\mathcal I}$, the point  $\crenn f(0)-1$ is detached from $0$ and $-1$, and Lemma \ref{lem-Psi-estim} implies that the distortion of the map $\Psi$ on $[f^{q_{n+1}+q_n}(0), f^{q_{n+1}}(0)]$ is uniformly bounded for $f\in \mathcal I$.  Thus $$| \Psi'|_{P\Psi^{-1}(x)}| >\frac c{M_n},$$ where $M_n=|I_n|$. It remains to prove the following:

 \begin{proposition}
 \label{prop-Pa}
In the family $f_a=f+a$, we have   $ P'_a(y) > c\frac{M_n}{|J_n|}$ for $y\in  [f^{q_{n+1}+q_n}(0), f^{q_{n+1}}(0)]$.
 \end{proposition}
 \begin{proof}

 If $P(y) = f^{q_{n+1}}(y)$, then $y\in [f^{q_{n}+q_{n+1}}(0), 0]\subset I_n$. Hence  $f^l(y)\in f^l(I_{n})$ and
 $$(f^{q_{n+1}-l})' (f^l(y)) >c \frac{|[f^{q_{n+1}+q_n}(0), f^{q_{n+1}}(0)]|}{|f^l(I_n)|}$$
 due to Theorem \ref{th-realbounds}b. The denominator is $|J_n|$, the numerator is commensurable with $M_n$ as explained above. Thus $(f^{q_{n+1}-l})' (f^l(y)) >c\frac{M_n}{|J_n|}$, and since other summands of the sum \eqref{eq-sum1} are positive, the statement follows.

 If $P(y) = f^{q_n+q_{n+1}}(y)$, then $y\in [0, f^{q_{n+1}}(0)]$. We have  $f^{q_n}(y)\in I_n$, thus  $f^{q_n+l}(y)\in f^l(I_{n})$ and $$(f^{q_{n+1}-l})' (f^{l+q_n}(y)) >c\frac{[f^{q_{n+1}+q_n}(0), f^{q_{n+1}}(0)]|}{|f^l(I_n)|}> \tilde c\frac{M_n}{|J_n|}.$$ Since other summands of the sum \eqref{eq-sum2} are positive, the statement follows.

 \end{proof}

The proof of Theorem~\ref{th-dR-estim} is thus completed.
\end{proof}

Now, we are ready to complete the proof of uniform hyperbolicity of $\mathcal R_{cyl}$ by providing an unstable invariant cone field.

\begin{theorem}
\label{th-cones}
 There exists a choice of $N$ in $\crenN\equiv\cren$ and a universal constant $c<1$  such that the cone field $$\mathcal C_f=\{v\in \mathcal D^{cr}_\eps \mid
\, |L_f(v) |>c\|v -  L_f(v) \|\}$$ defined on a neighborhood of $\mathcal I$ is invariant under $d\cren$ and vectors in the cones are uniformly expanding under $d\cren$.
\end{theorem}
Since linear functionals $L_f$ depend continuously on $f$  in weak topology, this cone field is continuous.
\begin{proof}
Choose $\lambda>1$.

Note that in Theorem \ref{th-dR-estim}, the interval $J_n$ is the smallest out of $q_{n+1}-2$ non-intersecting intervals, thus  $|J_n|<\frac 1{q_{n+1}-2}$. The denominators $q_{n}$ increase at a uniform exponential rate, thus we can choose $N$ so that $$\|d|_f\crenN \, 1\|>c\frac{1}{|J_n|}>\lambda\text{ for all  }f\in \mathcal I.$$

Increase $N$ if needed so that  for all $f\in \mathcal I$, the operator $d|_f\crenN$ uniformly contracts on a stable distribution: for any $w\in T \mathcal D^{cr}_{\eps}$ with $L_f(w)=0$, we have  $\|d_f\crenN w\|< \tau \|w\| $ for a universal  $\tau<1$. This is possible due to \cite[Theorem 6.4]{Ya4}.

From now on, we will write $\cren$ instead of $\crenN$ and $d\cren$ instead of $d|_f \cren$ for shortness.

 Note that we have $$L_{\cren f} (d\cren\xi)  = (d\cren \, 1) \cdot  L_f (\xi)$$  for any vector field $\xi\in T|_f \mathcal D^{cr}_{\eps}$, since this holds both for $\xi=1$ and for any $\xi\in \mathrm{ Ker}\,  L_f$.

Take $v\in T_f \mathcal D^{cr}_{\eps}$, and let  $k= L_f(v)$, then the vector field $v -  k=w$ belongs to the stable distribution, $L_f(w) =0$. Suppose that $v$ belongs to the cone  $\mathcal C_f$, i.e.  $c\|w\| <|k|$.
 Let us prove that $d\cren v$ belongs to the cone $\mathcal C_{\cren f}$ if $c$ was properly chosen. Indeed, $$|L_{\cren  f} (d\cren  v)| = \|d\cren  \, 1\| \cdot |L_f (v)| = \|d\cren  \, 1\| \cdot |k| $$ and $$\|d\cren v \|  = \|d\cren (w+k)\|\le\\ \tau \|w\| + |k|\cdot  \|d\cren  \, 1\|, $$ thus
 \begin{multline*}
\|d\cren v -L_{\cren  f} d\cren  v \|  \le\tau \|w\| + 2|k|\cdot \|d\cren  \, 1\|  \le \tau \frac {|k|}c + 2|k|\cdot \|d\cren  \, 1\|.
 \end{multline*}

It remains to find $c$ so that  $$\|d\cren  \, 1\| \cdot |k|> \tau |k| +2 c |k|\cdot  \|d\cren  \, 1\|.$$
 Since  $\|d\cren  \, 1\|>\lambda>1>\tau$, the inequality holds for small universal $c$.

 Now, let us show that the vectors in the cones are uniformly expanded. We use Theorem~\ref{th-dR-estim} again to choose $n$ so that
 $$\|d\cren^n \, 1\| = \|d{\mathcal R}_{\text{cyl}, nN} \, 1\|\gg 1+1/c\text{ for all }f\in \mathcal I.$$  Then for any $v\in \mathcal C_f$, using the representation $v=w+k$ again, we have $$\|d\cren ^n v\|> |k|\cdot  \|d\cren ^n \, 1\| - \tau \|w\|> |k|(  \|d\cren ^n\, 1\| -\tau/c)$$ and $\|v\|< |k|+\|w\| < |k| (1+1/c)$, thus $\|d\cren ^n v\| > C \|v\|$ for $C>1$, which implies uniform expansion.
\end{proof}

Standard techniques now imply uniform hyperbolicity of $\cren $:
\begin{itemize}
\item  the unstable distribution is constructed as an intersection of images of unstable cones;
\item This intersection may only be generated by one vector at each point  due to contraction in the transversal direction, and the resulting distribution  is unstable since it belongs to unstable cones;
\item the distribution depends continuously on a point since this is true for images of cones under $d\cren ^k$ for each $k$;
\item uniform transversality of stable and unstable distribution follows from the fact that the unstable vector $v$ at each point belongs to the unstable cone:  $|L_f(v) |>c\|v -  L_f(v) \|$, so if $v$ is normalized by $\|v\|=1$, then $L_f(v)$ is bounded away from zero.
\end{itemize}

Below we obtain finer estimates on the minimal expansion rate along the unstable direction of $\cren $, as required for Proposition \ref{prop-exponents} and Theorem \ref{th-smoothness-htype2}.
The next lemma shows that the partition $\mathcal P_n$ contains an interval $f(I_n)$ of length $\sim M_n^3$, which provides an upper estimate on the length of the shortest interval in $\mathcal P_n$.

\begin{lemma}
\label{lem-intervals2}
 For any critical map $f$ with irrational rotation number,  $|f(I_n)| <  c(f) M_n^{3}$ where $c(f)$ does not depend on $n$.

This bound is  universal on the invariant horseshoe.
\end{lemma}
\begin{proof} Since $|I_n|\to 0$ with $n$ and any cubic critical map has the Taylor series expansion $c_0+c_1x^3+\dots $ at zero,
for any critical map $f$  we have $|f(I_n)| < c(f) |I_n|^3 = c(f) M_n^3$ where $c(f)$ does not depend on $n$.

On the invariant horseshoe $\mathcal I$, there is a uniform estimate on $c(f)$ due to compactness of $\overline {\mathcal I}$.

\end{proof}

 Now we are ready to provide finer estimates on the expansion rates of the renormalization operator. Recall that $\delta\in(0,1)$ is a constant that satisfies   $M_n<C\delta^n$ for any $f\in \mathcal I$, see  Theorem \ref{th-intervals}.

Let $M_{m, n}$ be the length of the interval $[0, g^{{q_n}}(0)]$ that corresponds to the map $g=\mathcal R_{cyl, N}^m f$.

\begin{lemma}
\label{lem-Lyap-lowtype}
  For any irrational number  $\rho$, if $f\in \mathcal I$ has rotation number $\rho$, then  the minimal expansion rate of $\mathcal R_{cyl, N}$ over its unstable distribution $\eta_f$ along the orbit of $f$  is bounded below: $$\lambda_1\ge \left(\liminf_{n\to\infty} \inf_{m\ge 0} \sqrt[n]{\frac 1 {M_{m,n}^3}}\right)^N.$$ In particular,  $$\lambda_1\ge  \delta^{-3N}.$$ 
\end{lemma}
\begin{proof}
The minimal expansion rate can be computed using any vector in the unstable cones, since the difference of such vectors (properly rescaled) belongs to the stable distribution. Thus $$\lambda_1 = \liminf_{n\to \infty} \inf_{m\ge 0}\sqrt[n]{\|d \cren^n|_{\cren^m f} 1\|}.$$   Let $g=\cren^m f$. This norm is estimated below by the value of the linear functional $L_{\cren^n  g}$ on the vector field $d \cren^n \, 1$, since $L_{\cren g}$ is given by the integral with respect to the probability measure. Due to Corollary \ref{lem-intervals2}, we have $|J_n| \le c|M_n|^3$, thus Theorem \ref{th-dR-estim} implies  $$\|d\cren^n|_g \,1\| > \frac{c}{M_{m, Nn}^3}\text{ for all }n.$$ This implies the first statement.

Due to Theorem \ref{th-intervals}, we have $M_{m,n}< C\delta^{n}$. This implies the second statement.
\end{proof}

\begin{remark}[Periodic rotation numbers]
\label{rem-periodic}
It is easy to prove that for any periodic orbit of $\cren $, there exists a limit $\lim_{n\to\infty} \sqrt[n]{M_{m,n}} $ that is uniform with respect to $m$. Thus  $\liminf_{n\to\infty}\inf_m$ can be replaced with the limit in the previous estimate.

Due to Lemma \ref{lem-distributions}, the maximal expansion rate along the second unstable direction of $\mathcal R$ is $$\lambda_2 = \limsup_{n\to \infty} \sup_{m\ge 0} \sqrt[n]{(\Psi'_{m, n}(0))^2}$$ where $\Psi_{m,k}$ is the Douady coordinate for the $n$-th renormalization of $g=\cren^m f$. The estimate on the distortion of $\Psi$ from Lemma \ref{lem-Psi-estim} implies that $$\lambda_2 \le \left( \limsup_{n\to\infty} \sup_{m\ge 0} \sqrt[n]{M_{m,nN}^{-2}}\right)^N.$$
Again, for a periodic orbit we can replace $\limsup$ by a limit. So the inequality  $\lambda_1>\lambda_2$ in Proposition \ref{prop-exponents} will be satisfied for any periodic point of $\cren $.

 However, the previous estimates only guarantee the fraction of logarithms $\log \lambda_1/\log\lambda_2 =  3/2<2$, and we cannot conclude using Theorem \ref{th-smoothness} that the corresponding Arnold tongues are more than $C^1$ smooth. This agrees with the result of \cite{LlaveLuque}: for some periodic rotation numbers, numerical experiments showed that Arnold tongues are less than $C^2$ smooth.
\end{remark}

\begin{lemma}
\label{lem-Lyap-htype}
Let $\rho = [a_1, a_2, \dots].$  For any $A$, $\delta\ge 0$, for any  sufficiently large $K$,   if $\rho=[a_1, a_2,\dots]$ is such that $$ \quad \quad \frac{\#\{k \mid a_k<K, m\le k<m+n\}}{n}\le \delta $$
 for all $m$ and for all sufficiently large $n$, then for any $f\in \mathcal I$ with rotation number $\rho$, the minimal expansion rate of $\mathcal R_{cyl, N}$ over its unstable distribution $\eta_f$ along the orbit of $f$ is  $\lambda_1>A^N$.
\end{lemma}
\begin{proof}
Again, the minimal expansion rate can be computed using any vector in the unstable cones; we will use $v=1$.
Fix $A$, $\delta$.


By increasing $K$, we can guarantee that for any irrational $\rho$ and any large $n$, if at most the $\delta$-proportion of terms  $a_j$, $m\le j\le m+n N$, are smaller than $K$, then the denominators $q_{nN}$ of the continued fractional expansion of $G^m(\rho)$ grow quicklier than any given geometric progression. In particular, for any large $n$ and sufficiently large $K$, we have $q_{nN+1}> A^{nN}/c+2$ where $c$ is the same as in Theorem \ref{th-dR-estim}. Let $g=\cren^m f$, and apply Theorem \ref{th-dR-estim} to $g$. Since $J_{nN}$ is the shortest out of $q_{nN+1}-2$ non-intersecting intervals $g^l(I_{nN})$, $0<l<q_{nN+1}$, on the circle, we have  $|J_{nN}|<1/(q_{nN+1}-2) < c/A^{nN} $.
Thus for sufficiently large $n$, for all $m$, we have $\|d\cren^n|_g \, 1\| >  A^{nN}$ due to Theorem \ref{th-dR-estim}.


This implies the required estimate on the minimal expansion rate.

\end{proof}

Due to Lemma \ref{lem-distributions}, the minimal expansion rate along the unstable distribution of $\mathcal R|_{D_0}$ coincides with that of $\cren $, while the maximal expansion rate along the other unstable direction of $\mathcal R$ is bounded by the maximal value of  $(\Psi'_N(0))^{2}$ over $\overline{ \mathcal I}$. Thus Lemma \ref{lem-Lyap-htype} implies Proposition \ref{prop-exponents}
and Theorem \ref{th-smoothness-htype2}.

\bibliographystyle{amsalpha}
\bibliography{biblio}

\end{document}